\documentclass[reqno]{amsart} 
\usepackage[letterpaper, margin=1in]{geometry} 
\usepackage{amsmath}
\usepackage{amssymb}
\usepackage{hyperref}

\DeclareMathOperator{\supp}{supp}
\newcommand{\symdiff}{\mathbin{\triangle}}
\newcommand{\R}{{\mathbb R}}
\newcommand{\N}{{\mathbb N}}
\newcommand{\LL}{{\mathcal L}}
\newcommand{\Gcal}{{\mathcal G}}
\newtheorem{teo}{Theorem}[section]
\newtheorem{lema}[teo]{Lemma}
\newtheorem{prop}[teo]{Proposition}
\newtheorem{cor}[teo]{Corollary}

\theoremstyle{remark}
\newtheorem{obs}[teo]{Remark}

\theoremstyle{definition}
\newtheorem{defi}[teo]{Definition}
\newtheorem{ejem}[teo]{Example}

\numberwithin{equation}{section}
\begin{document}
\title[Optimal design for the fractional $g$--Laplacian]{Optimal design problems for the fractional $g$--Laplacian at prescribed energy level}
\author{Juan F. Spedaletti}

\address{Departamento de Matem\'atica, FCFMyN, Universidad Nacional de San
Luis and Instituto de Matem\'atica Aplicada San
Luis (IMASL), CONICET--UNSL, Av. Italia 1556, office
5, San Luis (5700), San Luis, Argentina.}

\email{jfspedaletti@unsl.edu.ar}

\urladdr{https://juanspedaletti.github.io}
\subjclass[2020]{35R11; 35P30; 49Q10; 46E30}

\keywords{Optimal design, shape optimization, fractional $g$--Laplacian, fractional Orlicz--Sobolev spaces, nonlinear eigenvalue problems}

\begin{abstract}
In this work we study two optimal design problems, of hard and soft obstacle type, in fractional Orlicz--Sobolev spaces. We also establish the connection between these problems, and finally we analyze the asymptotic behavior of the quotients $\Lambda_{s,\mu}(\alpha)/\mu$ with respect to the energy level $\mu$ at the endpoints $0$ and $+\infty$.
\end{abstract}
\maketitle
\section{Introduction}
Optimal design problems for eigenvalue-type constants have been studied in several settings. In the local one, the question is how to place an obstacle, or a hole, of prescribed measure inside a given domain so as to optimize an eigenvalue of the associated operator. The first Steklov eigenvalue of domains with holes is optimized in \cite{FBGR} by means of a shape derivative approach, a constrained shape optimization problem in Orlicz--Sobolev spaces is studied in \cite{dSSSS}, and the local counterpart of the soft obstacle problem considered below --- an eigenvalue of the $g$--Laplacian penalized on a set of prescribed measure, together with the limit of large penalization --- is treated in \cite{SSS}.

In the nonlocal setting, optimal design problems for the eigenvalues of the fractional $p$--Laplacian were studied in \cite{FBS}, where the hard and the soft obstacle problems are introduced, the optimal obstacle is identified with the set where an extremal vanishes, the two problems are connected by letting the penalization parameter tend to infinity, and the behavior as $s\uparrow1$ is analyzed; see also \cite{FBRS} for the first $p$--fractional eigenvalue with mixed boundary conditions. The present work extends the results of \cite{FBS} to the fractional Orlicz--Sobolev setting.

The functional framework for that extension was developed in \cite{FBSalort}, where fractional Orlicz--Sobolev spaces are introduced and their basic properties are obtained; see \cite{ACPS} for the corresponding embeddings and \cite{BahrouniSalort} for Neumann and Robin boundary conditions. Eigenvalue problems for the fractional $g$--Laplacian were considered in \cite{Salort} and, in the case without a set constraint, in \cite{FBSalortEig}, while regularity properties of the eigenfunctions and of the solutions can be found in \cite{FBSalortVivas} and \cite{KLL}. The essential new feature of this setting is the lack of homogeneity of the modular, which prevents normalizing the admissible functions and forces the problems to be posed at a prescribed energy level; as a consequence, the constants depend on that level in a nontrivial way, and this dependence is one of the subjects of the present paper.

Let $\Omega\subset \R^n$ be an open, connected and bounded domain. For $0<s<1$ and $G$ a Young function we define the modular
$$
\Phi_{s,G}(u):=\iint_{\R^{2n}}G(D^su)\,d\nu_n,
$$
where
$$
D^su=\frac{u(x)-u(y)}{|x-y|^s}
$$
is the \emph{H\"older quotient}. In contrast with the power case $G(t)=t^p$, the modular $\Phi_{s,G}$ is not homogeneous, and therefore the minimization problems associated with it have to be posed at a prescribed energy level. Given $\mu>0$, we normalize the admissible functions by
$$
\Gcal(u):=\int_\Omega G(u)\,dx=\mu .
$$
We refer to Section \ref{sec:prelim} for the precise definitions of the measure $\nu_n$ and of the spaces $W^{s,G}_0(\Omega)$ and $W^{s,G}_A(\Omega)$.

\medskip

\noindent\textbf{The $(s,G)$--hard obstacle problem.} Given a measurable set $A\subset\Omega$ with $|A|>0$, the admissible functions are forced to vanish on $A$, and we consider the constant
\begin{equation}\label{lambdamu}
\lambda_{s,\mu}(A):=\inf\left \{\iint_{\R^{2n}}G(D^s u)\,d\nu_n\colon u\in W_A^{s,G}(\Omega),\ \Gcal(u)=\mu \right \}.
\end{equation}
Since $\Phi_{s,G}$ is not homogeneous, $\lambda_{s,\mu}(A)$ is not the optimal constant of a Poincar\'e-type inequality, as is the case when $G(t)=t^p$; it is instead a constrained minimum, and it is convenient to regard it as a \emph{shape functional}
$$
\lambda_{s,\mu}\colon\mathcal{A}\longrightarrow(0,+\infty],\qquad
\mathcal{A}:=\{A\subset\Omega\colon A \text{ measurable}\},
$$
with the convention $\inf\emptyset=+\infty$.

\begin{obs}\label{obs:funcional}
The functional $\lambda_{s,\mu}$ has the following elementary properties.
\begin{enumerate}
\item It is well defined on equivalence classes of sets modulo null sets: if $|A\triangle A'|=0$ then $W_A^{s,G}(\Omega)=W_{A'}^{s,G}(\Omega)$ and hence $\lambda_{s,\mu}(A)=\lambda_{s,\mu}(A')$.
\item $\lambda_{s,\mu}(A)>0$ for every $A\in\mathcal{A}$. Indeed, $\lambda_{s,\mu}(A)\ge\mu/C$ with $C=C(n,s,G,\Omega)$, by the Poincar\'e-type inequality of Proposition \ref{prop:poincare} below. The value $+\infty$ is attained precisely when no $u\in W_A^{s,G}(\Omega)$ satisfies $\Gcal(u)=\mu$. Let us stress that, in contrast with the regional setting of \cite{FBS}, this positivity comes from the exterior condition and not from the obstacle, so that it holds for $A=\emptyset$ as well; the obstacle provides in addition a lower bound proportional to $|A|$. Indeed, since $u=0$ a.e. in $A$, restricting the integration in $\Phi_{s,G}(u)$ to $\Omega\times A$ gives $\lambda_{s,\mu}(A)\ge c\,|A|\,\mu$ with $c=c(n,s,G,\Omega)>0$, which is the analogue of \cite[Proposition 2.10]{FBS}.
\item It is nondecreasing with respect to inclusion: if $A\subset A'$ then $W_{A'}^{s,G}(\Omega)\subset W_{A}^{s,G}(\Omega)$ and therefore $\lambda_{s,\mu}(A)\le\lambda_{s,\mu}(A')$. In particular, some constraint on the size of the admissible obstacles is needed for the minimization to be nontrivial.
\item For fixed $A$, the map $\mu\mapsto\lambda_{s,\mu}(A)$ is nondecreasing. Indeed, if $0<\mu_1<\mu_2$ and $\Gcal(u)=\mu_2$, choosing $0<t<1$ with $\Gcal(tu)=\mu_1$ we get $\lambda_{s,\mu_1}(A)\le\Phi_{s,G}(tu)\le\Phi_{s,G}(u)$, and it suffices to take the infimum over such $u$.
\end{enumerate}
\end{obs}

The set $A$ plays the role of an obstacle on which the admissible functions must vanish. The first optimal design problem that we address consists in minimizing this constant among obstacles of prescribed measure: given $\alpha\in(0,1)$ we set
$$
\mathcal{A}_\alpha:=\{A\subset\Omega\colon A\text{ measurable},\ |A|=\alpha|\Omega|\},
$$
and we look for an optimal configuration $A\in\mathcal{A}_\alpha$ such that
\begin{equation}\label{Lambdamu}
\lambda_{s,\mu}(A)=\Lambda_{s,\mu}(\alpha):=\inf\{\lambda_{s,\mu}(A')\colon A'\in\mathcal{A}_\alpha\}.
\end{equation}
We call \eqref{Lambdamu} the \emph{$(s,G)$--hard obstacle problem at energy level $\mu$}, since the optimal set $A$ can be seen as the obstacle where the extremals are forced to vanish.

\medskip

\noindent\textbf{The $(s,G)$--soft obstacle problem.} A natural variant is obtained by replacing the constraint $u=0$ on $A$ by a penalization term. Since the class of characteristic functions is not closed under any reasonable topology, we relax it and consider the admissible potentials
$$
\mathcal{B}_\alpha:=\left\{\varphi\in L^\infty(\Omega)\colon 0\le\varphi\le 1,\ \int_\Omega \varphi\,dx=\alpha|\Omega|\right\},
$$
so that $\chi_A\in\mathcal{B}_\alpha$ if and only if $A\in\mathcal{A}_\alpha$. Given $\sigma>0$ and $\varphi\in\mathcal{B}_\alpha$ we define
\begin{equation}\label{lambdasoft}
\lambda_{s,\mu}(\varphi,\sigma):=\inf\left\{\iint_{\R^{2n}}G(D^su)\,d\nu_n+\sigma\int_\Omega G(u)\varphi\,dx\colon u\in W_0^{s,G}(\Omega),\ \Gcal(u)=\mu\right\}.
\end{equation}
The corresponding optimal design problem consists in finding an optimal potential $\varphi_\sigma\in\mathcal{B}_\alpha$ such that
\begin{equation}\label{Lambdasoft}
\lambda_{s,\mu}(\varphi_\sigma,\sigma)=\Lambda_{s,\mu}(\alpha,\sigma):=\inf\{\lambda_{s,\mu}(\varphi,\sigma)\colon \varphi\in\mathcal{B}_\alpha\}.
\end{equation}
We call \eqref{Lambdasoft} the \emph{$(s,G)$--soft obstacle problem at energy level $\mu$}.

\begin{obs}\label{obs:funcionalsoft}
The constants \eqref{lambdasoft} enjoy the following properties, which parallel those of Remark \ref{obs:funcional}.
\begin{enumerate}
\item They depend on $\varphi$ only through its equivalence class modulo null sets.
\item They are finite and positive. More precisely,
$$
\frac{\mu}{C}\ \le\ \lambda_{s,\mu}(\varphi,\sigma)\ \le\ \lambda_{s,\mu}(\emptyset)+\sigma\mu,
$$
where $C=C(n,s,G,\Omega)$ is the constant of Proposition \ref{prop:poincare} below. The lower bound follows by discarding the nonnegative penalization term and applying Proposition \ref{prop:poincare}, and the upper one from $0\le\varphi\le1$, which gives $\int_\Omega G(u)\varphi\,dx\le\Gcal(u)=\mu$. In particular, and in contrast with $\lambda_{s,\mu}(A)$, the value $+\infty$ is never attained, the reason being that here the class of admissible functions does not depend on $\varphi$. The same bounds hold for $\Lambda_{s,\mu}(\alpha,\sigma)$.
\item They are nondecreasing in $\varphi$ and in $\sigma$: if $\varphi\le\varphi'$ a.e. in $\Omega$ then $\lambda_{s,\mu}(\varphi,\sigma)\le\lambda_{s,\mu}(\varphi',\sigma)$, and if $\sigma\le\sigma'$ then $\lambda_{s,\mu}(\varphi,\sigma)\le\lambda_{s,\mu}(\varphi,\sigma')$.
\item For fixed $\sigma$ and $\varphi$, the map $\mu\mapsto\lambda_{s,\mu}(\varphi,\sigma)$ is nondecreasing, by the same scaling argument as in Remark \ref{obs:funcional}.
\item For every $A\in\mathcal{A}_\alpha$ and every $\sigma>0$,
$$
\lambda_{s,\mu}(\chi_A,\sigma)\le\lambda_{s,\mu}(A),
$$
since every $u\in W_A^{s,G}(\Omega)$ with $\Gcal(u)=\mu$ is admissible in \eqref{lambdasoft} and the penalization term vanishes for such a $u$. Together with (3), this shows that the limit of $\lambda_{s,\mu}(\chi_A,\sigma)$ as $\sigma\to\infty$ exists and does not exceed $\lambda_{s,\mu}(A)$.
\end{enumerate}
\end{obs}

\medskip

The two problems are related by the fact that the term $\sigma\int_\Omega G(u)\varphi\,dx$ acts as a penalization, so that heuristically
$$
\lambda_{s,\mu}(\chi_A,\sigma)\longrightarrow\lambda_{s,\mu}(A)\qquad\text{as }\sigma\to\infty,
$$
and one of our aims is to make this assertion rigorous, both at the level of the constants and at the level of the optimal configurations. By Remark \ref{obs:funcionalsoft}(5) one of the two inequalities is immediate, so that the content of this assertion is the reverse one.

\medskip

Sections \ref{sec:prelim} and \ref{sec:operator} collect the preliminaries on Young functions, Orlicz and fractional Orlicz--Sobolev spaces, on the operator $(-\Delta_g)^s$ and on the two eigenvalue problems that arise as the Euler--Lagrange equations of the constants above.

Section \ref{sec:optimal} is devoted to the two optimal design problems. For the hard obstacle problem we first characterize $\Lambda_{s,\mu}(\alpha)$ as an infimum taken over functions whose zero set is large enough, with no reference to the obstacle (Lemma \ref{lema:caract}), and deduce from it the existence of an optimal configuration (Theorem \ref{exop}). The main result of the section, and the first main result of the paper, is Theorem \ref{teo:identif}: the zero set of an extremal has exactly the prescribed measure,
$$
\big|\{u=0\}\cap\Omega\big|=\alpha|\Omega| ,
$$
from which the identification of the optimal obstacle with the set where the corresponding extremal vanishes follows at once (Corollary \ref{coro:identif}). In other words, the optimal configuration is not an additional unknown: it is determined by any extremal. Let us stress that its proof uses neither a minimum principle nor any regularity theory for $(-\Delta_g)^s$; the only ingredients are the Euler--Lagrange equation and the fact that the layer $\{0<u<\delta\}$ shrinks to a null set as $\delta\to0^+$. For the soft obstacle problem we prove, by means of the bathtub principle, that there exists an optimal pair and that the optimal potential can always be taken to be a characteristic function (Theorem \ref{teo:paroptimo}), so that nothing is lost in the relaxation. Finally, Theorem \ref{teo:conexion} makes the above heuristics rigorous: the constants of the soft problem increase to the constant of the hard one as $\sigma\to\infty$ and, up to a subsequence, the optimal pairs converge to an optimal configuration of the hard problem, the optimal potentials converging to the characteristic function of the optimal obstacle.

Section \ref{sec:asymptotic} analyzes the dependence on the energy level $\mu$. Since $\Phi_{s,G}$ is not homogeneous, the constants at two different energy levels can only be compared up to the exponents $p^-/p^+$ and $p^+/p^-$ (Proposition \ref{prop:escala}), and from that comparison we obtain that $\mu\mapsto\Lambda_{s,\mu}(\alpha)$ is strictly increasing and locally Lipschitz continuous (Corollary \ref{propLambda}). The main result of the section is Theorem \ref{limitecociente}, which identifies the limits of the quotients $\Lambda_{s,\mu}(\alpha)/\mu$ at the two ends of the energy scale: under a non-concentration condition on the extremals, and for domains that are star-shaped with respect to a ball, those limits are the optimal design constants, for the same obstacle problem, associated with the fractional $p_0$-- and $p_\infty$--Laplacians, where $p_0$ and $p_\infty$ are the Matuszewska exponents of $G$ at the origin and at infinity. The theorem also shows that the suitably rescaled extremals converge to extremals of the corresponding limit problem, and Example \ref{ejem:distintos} exhibits a Young function and a domain for which the two limits are different, so that the two limit problems do not coincide. As a consequence, in Theorem \ref{converoptimos} we prove that the optimal sets $A_\mu$ converge in $L^1(\Omega)$, up to a subsequence, to an optimal set of the corresponding limit problem.

\section{Preliminaries}\label{sec:prelim}
\subsection{Young functions and Orlicz spaces}

A function $G\colon \R\to\R$ is called a {\em Young function} if it can be written as
$$
G(t) = \int_0^{|t|} g(r)\, dr,
$$
where $g\colon [0,\infty)\to [0,\infty)$ is nondecreasing and right continuous, $g(t)=0$ if and only if $t=0$, and $g(t)\to\infty$ as $t\to\infty$. Then $G$ is even, convex and continuous, $G(t)>0$ for $t\neq0$, $G(t)/t\to 0$ as $t\to 0$ and $G(t)/t\to\infty$ as $t\to\infty$; moreover $g=G'$ a.e. Since $G$ is even, throughout the paper we simply write $G(t)$ instead of $G(|t|)$.

It will be helpful to extend $g$ to the entire real line by oddness, setting $g(0)=0$ and $g(t)=-g(-t)$ for $t<0$; from now on, $g$ is always understood to be extended to $\R$ in this way.
We recall now some basic definitions on Orlicz spaces that can be found, for instance, in \cite{KR}.

Let $U\subset \R^n$ be an open set and let $\nu$ be a $\sigma$--finite Borel measure on $U$. The Orlicz class $\LL_G(U, d\nu)$ is defined as
$$
\LL_G(U,d\nu) := \left\{u\colon U\to \R,\ \text{measurable}\colon \int_U G(u)\, d\nu < \infty\right\}.
$$
The Orlicz space $L_G(U,d\nu)$ is then defined as the linear hull of $\LL_G(U,d\nu)$. It follows that $L_G(U,d\nu)$ can be characterized as
$$
L_G(U,d\nu) = \left\{ u\colon U\to \R,\ \text{$\nu$--measurable}\colon \int_U G\left (\frac{u}{k}\right )\, d\nu < \infty, \text{ for some } k>0\right\}.
$$
This space is a Banach space when it is equipped, for instance, with the {\em Luxemburg norm}, i.e.
$$
\|u\|_{L_G(U,d\nu)} = \|u\|_{G,U,d\nu} = \|u\|_{G,d\nu} := \inf\left\{ k>0\colon \int_U G\left(\frac{u}{k}\right)\, d\nu \leq 1\right\}.
$$
We shall also use the following classes of Young functions,
$$
\Delta_2^0:=\Big\{G\colon \limsup_{t\to 0}\frac{G(2t)}{G(t)}<\infty\Big\},\qquad
\Delta_2^\infty:=\Big\{G\colon \limsup_{t\to \infty}\frac{G(2t)}{G(t)}<\infty\Big\},
$$
and we set $\Delta_2:=\Delta_2^0\cap\Delta_2^\infty$. Since $G$ is continuous and positive on $(0,+\infty)$, the quotient $G(2t)/G(t)$ is bounded on compact subsets of $(0,+\infty)$, and therefore
\begin{equation}\label{Delta2}
G\in\Delta_2\qquad\Longleftrightarrow\qquad \sup_{t>0}\frac{G(2t)}{G(t)}<\infty .
\end{equation}
A well known and interesting fact is that $\LL_G(U,d\nu) = L_G(U,d\nu)$ if and only if $G\in\Delta_2$, and that the Orlicz space $L_G(U,d\nu)$ is separable under the same condition. Let us stress that it is the global condition \eqref{Delta2} which is relevant here: its classical formulation in terms of $\Delta_2^\infty$ alone is equivalent to \eqref{Delta2} only for finite measures, and the measures we shall work with are not finite.

Given a Young function $G$, we define its complementary function $\bar G$ as
$$
\bar G(t) := \sup\{ \tau|t|-G(\tau)\colon \tau\geq 0\}.
$$
Observe that $\bar G$ is also a Young function and is the optimal function in the Young inequality
\begin{equation}\label{young}
\tau t\le G(t) + \bar G(\tau),
\end{equation}
for all $\tau,t$. Observe that equality in \eqref{young} is achieved if and only if $\tau=g(t)\operatorname{sign} t$ or $t=\bar g(\tau)\operatorname{sign}\tau$ where $\bar{g}(t)$ is the derivative of $\bar{G}(t)$.

It follows directly from \eqref{young} that if $u\in L_G(U,d\nu)$ and $v\in L_{\bar G}(U,d\nu)$, then $uv\in L^1(U)$ and
$$
\int_U |uv|\, d\nu \le 2 \|u\|_G \|v\|_{\bar G}.
$$
We recall that a Young function $G$ satisfies the Lieberman condition \cite{Lieberman} if
\begin{equation}\label{Liebermann}\tag{L}
1<p^-\leq\frac{tg(t)}{G(t)}\leq p^+<\infty\qquad \forall t>0.
\end{equation}
Observe that the convexity of $G$ gives $G(t)=\int_0^tg(s)\,ds\le tg(t)$, so that the quotient in \eqref{Liebermann} is always greater than or equal to $1$; the content of the left hand inequality is therefore the strict inequality $p^->1$.

The condition \eqref{Liebermann} is equivalent to
\begin{equation}\label{sandwich}
\min\{a^{p^-},a^{p^+}\}\,G(t)\le G(at)\le \max\{a^{p^-},a^{p^+}\}\,G(t),\qquad a,t>0,
\end{equation}
see \cite[Lemma 2.1]{FIN} or \cite[Lemma 2.1]{Salort}.

The Lieberman condition has a simple reading in terms of the classes introduced above. By \cite[Theorem 4.1]{KR}, $G\in\Delta_2$ if and only if there exists $p>1$ such that $tg(t)/G(t)\le p$ for every $t>0$; that is, the right hand side of \eqref{Liebermann} amounts to $G\in\Delta_2$. The left hand side corresponds, in the same way, to $\bar G\in\Delta_2$: from the equality case in \eqref{young} we have $\bar G(g(t))=tg(t)-G(t)$ for every $t>0$, so that, writing $s=g(t)$ and $r:=tg(t)/G(t)$,
$$
\frac{s\bar g(s)}{\bar G(s)}=\frac{r}{r-1},
$$
and $r\mapsto r/(r-1)$ is decreasing; hence $G$ satisfies \eqref{Liebermann} with exponents $p^-,p^+$ if and only if $\bar G$ satisfies \eqref{Liebermann} with the conjugate exponents $(p^+)'$ and $(p^-)'$. Summarizing,
\begin{equation}\label{LyDelta2}
\text{\eqref{Liebermann} holds}\qquad\Longleftrightarrow\qquad G\in\Delta_2\ \text{ and }\ \bar G\in\Delta_2 ,
\end{equation}
which is the reason why \eqref{Liebermann} is the natural standing assumption in this paper.

In order to analyze the asymptotic behavior of the problems we consider the so called Matuszewska--Orlicz functions, defined for $t>0$ as
\begin{equation}\label{Matu}
M_0(t):=\limsup_{\varepsilon\to 0}\frac{G(\varepsilon t)}{G(\varepsilon)}\qquad\text{and}\qquad M_\infty(t):=\limsup_{\varepsilon\to \infty}\frac{G(\varepsilon t)}{G(\varepsilon)}.
\end{equation}
These functions were introduced in \cite{MO}; see also \cite{Mat} and the monograph \cite{Maligranda}. For $i=0,\infty$ the function $M_i$ is convex and submultiplicative, that is $M_i(\tau t)\le M_i(\tau)M_i(t)$, and satisfies $M_i(0)=0$ and $M_i(1)=1$. Moreover $M_i(t)<\infty$ for every $t>0$ whenever $G\in\Delta_2^i$, and if in addition the $\limsup$ in \eqref{Matu} is a limit, then there exists $p_i\ge1$ such that $M_i(t)=t^{p_i}$; see \cite[Section 2]{FBSalortEig} and, for the last assertion, \cite[Proposition 2.2]{FBSalortEig}.

Under \eqref{Liebermann} these facts are available for free. Indeed $G\in\Delta_2$ by \eqref{LyDelta2}, and applying \eqref{sandwich} with $a=t$ and with $\varepsilon$ in place of $t$, dividing by $G(\varepsilon)$ and letting $\varepsilon\to i$, we obtain
\begin{equation}\label{cotasMatu}
\min\{t^{p^-},t^{p^+}\}\le M_i(t)\le\max\{t^{p^-},t^{p^+}\},\qquad t>0,\ i=0,\infty .
\end{equation}
In particular $M_0$ and $M_\infty$ are finite, and whenever they are powers the corresponding exponent satisfies $p_i\in[p^-,p^+]$.

\begin{obs}\label{obs:Matuyoung}
Assume \eqref{Liebermann}. Then $M_0$ and $M_\infty$ are Young functions satisfying \eqref{Liebermann} with the same exponents $p^-$ and $p^+$. Indeed, $M_i$ is convex, being a pointwise $\limsup$ of convex functions, \eqref{cotasMatu} gives the required normalizations at the origin and at infinity since $p^->1$, and the submultiplicativity of $M_i$ together with \eqref{cotasMatu} gives \eqref{sandwich} for $M_i$. In particular every result of this paper proved under \eqref{Liebermann} applies to $M_0$ and $M_\infty$ as well.
\end{obs}
\subsection{Fractional Orlicz-Sobolev spaces}
In the product space $\R^n\times\R^n = \R^{2n}$ we define the measure
$$
d\nu_n := \frac{dxdy}{|x-y|^n}.
$$
Observe that this is a Borel measure and that if $K\subset \R^{2n}\setminus \Delta$ is compact, then $\nu_n(K)<\infty$, where $\Delta\subset \R^n\times\R^n$ is the diagonal $
\Delta := \{(x,x)\colon x\in\R^n\}$.

We will consider Orlicz spaces $L_G(U,d\nu)$ in three situations: with $U=\Omega$ and with $U=\R^n$, in both cases for $d\nu=dx$ the Lebesgue measure, and with $U=\R^{2n}$ for $d\nu=d\nu_n$.

We will use the notations
$$
L_G(\R^n)=L_G(\R^n,dx),\qquad L_G(\Omega)=L_G(\Omega,dx),\qquad L_G(\nu_n)=L_G(\R^{2n},d\nu_n).
$$
Now, given a fractional parameter $s\in (0,1)$, we introduce the notation for the H\"older quotient of a function $u\colon\R^n\to \R$.
$$
D^su(x,y) := \frac{u(x)-u(y)}{|x-y|^s}.
$$
Then $D^su\colon \R^{2n}\setminus \Delta\to \R$. We consider the Orlicz-Sobolev space $W^{s,G}(\R^n)$ defined as
$$
W^{s,G}(\R^n)=\{u\in L_G(\R^n)\colon D^su\in L_G(\R^{2n},\nu_n)\},
$$
equipped with the norm
$$
\|u\|_{s,G}:=\|u\|_{L_G}+\|D^su\|_{G,\nu_n}.
$$
Also, we define the following versions that will be useful: for $\Omega\subset \R^n$ open and $A\subset \Omega$
\begin{align*}
W_0^{s,G}(\Omega)&=\{u\in W^{s,G}(\R^n)\colon u=0 \text{ a.e. in }\R^n\setminus \Omega\}
\\
W_A^{s,G}(\Omega)&=\{u\in W_0^{s,G}(\Omega)\colon u=0\text{ a.e. in }A\}.
\end{align*}
For $u\in W^{s,G}(\R^n)$ we write
$$
[u]_{s,G}:=\|D^su\|_{G,\nu_n}=\inf\left\{\lambda>0\colon \Phi_{s,G}\left(\frac{u}{\lambda}\right)\le 1\right\}
$$
for the \emph{Gagliardo--Luxemburg seminorm} of $u$, so that $\|u\|_{s,G}=\|u\|_{L_G(\R^n)}+[u]_{s,G}$.

The basic functional properties of these spaces were obtained in \cite{FBSalort}.
\begin{prop}\label{prop:espacios}
Let $0<s<1$ and let $G$ be a Young function. Then
\begin{enumerate}
\item $W^{s,G}(\R^n)$ is a Banach space;
\item if $G$ satisfies the $\Delta_2$--condition \eqref{Delta2}, then $W^{s,G}(\R^n)$ is separable;
\item if both $G$ and $\bar G$ satisfy the $\Delta_2$--condition, that is, if \eqref{Liebermann} holds, then $W^{s,G}(\R^n)$ is reflexive;
\item if $\Omega$ is bounded and \eqref{Liebermann} holds, then the embedding $W_0^{s,G}(\Omega)\hookrightarrow L_G(\Omega)$ is compact.
\end{enumerate}
\end{prop}

These properties are inherited by the spaces in which we shall work. Indeed, convergence in $L_G(\R^n)$ implies, up to a subsequence, convergence almost everywhere, so that the conditions $u=0$ a.e. in $\R^n\setminus\Omega$ and $u=0$ a.e. in $A$ are stable under limits; hence $W_0^{s,G}(\Omega)$ and $W_A^{s,G}(\Omega)$ are closed linear subspaces of $W^{s,G}(\R^n)$. In particular they are Banach spaces, they are separable whenever $G$ satisfies \eqref{Delta2}, and they are reflexive -- and, being convex, weakly closed -- whenever \eqref{Liebermann} holds. Moreover, if $\Omega$ is bounded and \eqref{Liebermann} holds, the embeddings $W_0^{s,G}(\Omega)\hookrightarrow L_G(\Omega)$ and $W_A^{s,G}(\Omega)\hookrightarrow L_G(\Omega)$ are compact. Let us stress that no regularity of $\partial\Omega$ is required here: the compactness result of \cite{FBSalort} is local in $\R^n$, and the elements of $W_0^{s,G}(\Omega)$ vanish outside the bounded set $\Omega$, so that local convergence in $L_G$ is in fact convergence in $L_G(\Omega)$. Let us finally mention that, under \eqref{Liebermann} and for $\Omega$ a bounded Lipschitz domain, $W_0^{s,G}(\Omega)$ coincides with the closure of $C_c^\infty(\Omega)$ in $W^{s,G}(\R^n)$; see \cite{FSV} for the case $G(t)=t^p$ and \cite{BBO} for the Orlicz setting.

Under \eqref{Liebermann} the modular and the Luxemburg norm control each other. Indeed, by \eqref{sandwich} for every $u\in L_G(\Omega)$
$$
\min\left\{\|u\|_{G,\Omega}^{p^-},\|u\|_{G,\Omega}^{p^+}\right\}\le \int_\Omega G(u)\,dx\le \max\left\{\|u\|_{G,\Omega}^{p^-},\|u\|_{G,\Omega}^{p^+}\right\}.
$$
In particular, modular convergence and norm convergence are equivalent, and the same holds in $L_G(\nu_n)$ for the seminorm $[\,\cdot\,]_{s,G}$ and the modular $\Phi_{s,G}$.

Finally, we shall use the following Poincar\'e-type inequality, which in particular guarantees that the constants introduced below are positive.
\begin{prop}\label{prop:poincare}
Let $\Omega\subset\R^n$ be open and bounded, and assume \eqref{Liebermann}. Then there exists a constant $C=C(n,s,G,\Omega)>0$ such that
$$
\int_\Omega G(u)\,dx\le C\,\Phi_{s,G}(u)\qquad\text{for every } u\in W_0^{s,G}(\Omega).
$$
\end{prop}
This is \cite[Proposition 3.2]{Salort}, where the inequality is obtained in the equivalent form $\int_\Omega G(u)\,dx\le \Phi_{s,G}(C\,u)$ with an explicit constant depending only on $n$, $s$, $p^\pm$ and $\operatorname{diam}(\Omega)$; the stated form follows from \eqref{sandwich}.

\begin{obs}\label{obs:normamodular}
Proposition \ref{prop:poincare} says that on $W_0^{s,G}(\Omega)$ --- and hence on its closed subspace $W_A^{s,G}(\Omega)$ --- the Gagliardo--Luxemburg seminorm is in fact a norm, equivalent to $\|\cdot\|_{s,G}$. Indeed, $[u]_{s,G}\le\|u\|_{s,G}$ always, while the norm counterpart of Proposition \ref{prop:poincare}, namely $\|u\|_{G,\Omega}\le C\,[u]_{s,G}$ for $u\in W_0^{s,G}(\Omega)$, is \cite[Corollary 3.3]{Salort}.
Moreover, by \eqref{sandwich},
$$
\min\big\{[u]_{s,G}^{p^-},[u]_{s,G}^{p^+}\big\}\le \Phi_{s,G}(u)\le \max\big\{[u]_{s,G}^{p^-},[u]_{s,G}^{p^+}\big\},
$$
so that a subset of $W_A^{s,G}(\Omega)$ is bounded if and only if the modulars of its elements are bounded, and $u_i\to u$ in $W_A^{s,G}(\Omega)$ if and only if $\Phi_{s,G}(u_i-u)\to0$. We shall use these facts repeatedly, often without further mention.
\end{obs}

The following lemma, which replaces the elementary inequality available for $G(t)=t^p$, will be used repeatedly in what follows.

\begin{lema}\label{lema:convG}
Let $U$ and $\nu$ be as above and let $M$ be a Young function satisfying \eqref{Liebermann}. If $u_k\to u$ in $L_M(U,d\nu)$, then $M(u_k)\to M(u)$ in $L^1(U,d\nu)$.
\end{lema}
\begin{proof}
By \cite[Lemma 2.6]{FBSalort}, for every $\eta>0$ there exists $C_\eta>0$ such that $M(a+b)\le C_\eta M(a)+(1+\eta)^{p^+}M(b)$ for all $a,b>0$, and hence, for every $\varepsilon>0$ there exists $C_\varepsilon>0$ such that
$$
\big|M(a+b)-M(a)\big|\le \varepsilon\,M(a)+C_\varepsilon\,M(b)\qquad\text{for all }a,b\in\R .
$$
Applying this inequality with $a=u$ and $b=u_k-u$ and integrating over $U$,
$$
\int_U\big|M(u_k)-M(u)\big|\,d\nu\ \le\ \varepsilon\int_U M(u)\,d\nu+C_\varepsilon\int_U M(u_k-u)\,d\nu .
$$
By \eqref{sandwich} modular and norm convergence are equivalent in $L_M(U,d\nu)$, so that the last integral tends to $0$ as $k\to\infty$. Therefore
$$
\limsup_{k\to\infty}\int_U\big|M(u_k)-M(u)\big|\,d\nu\le \varepsilon\int_U M(u)\,d\nu
$$
for every $\varepsilon>0$, and the conclusion follows by letting $\varepsilon\to0$.
\end{proof}

\section{The operator $(-\Delta_g)^s$}\label{sec:operator}

In this section we introduce the operator associated with the modular $\Phi_{s,G}$ and we discuss in which classes of functions it is well defined. We then state the two eigenvalue problems that will appear as the Euler--Lagrange equations of the constants introduced in Section \ref{sec:prelim}: the eigenvalue problem with obstacle, associated with $\lambda_{s,\mu}(A)$, and the eigenvalue problem with a potential, associated with $\lambda_{s,\mu}(\varphi,\sigma)$.

\subsection{The fractional $g$--Laplacian operator}\label{subsec:operador}

Given $0<s<1$, the \emph{fractional $g$--Laplacian} of a function $u\colon\R^n\to\R$ is formally defined as
\begin{equation}\label{glap}
(-\Delta_g)^su(x):=2\,\mathrm{p.v.}\int_{\R^n}g\big(D^su(x,y)\big)\,\frac{dy}{|x-y|^{n+s}}
=2\lim_{\varepsilon\downarrow0}\int_{\R^n\setminus B_\varepsilon(x)}g\big(D^su(x,y)\big)\,\frac{dy}{|x-y|^{n+s}},
\end{equation}
where $g$ is understood to be extended to $\R$ by oddness. The principal value in \eqref{glap} is needed, since in general the integrand is not absolutely integrable near the singularity.

The following lemma gives a class of functions for which \eqref{glap} makes sense pointwise; it is the analogue of \cite[Lemma 2.1]{FBS}, where the case $G(t)=t^p$ with $p\ge2$ is considered.

\begin{lema}\label{lema:bienbdef}
Let $0<s<1$ and assume that $g$ is locally Lipschitz in $\R$. If $u\in C^2(\Omega)\cap L^\infty(\R^n)$, then $(-\Delta_g)^su(x)$ is well defined for every $x\in\Omega$.
\end{lema}
\begin{proof}
Fix $x\in\Omega$ and let $\varepsilon_0>0$ be such that $B_{\varepsilon_0}(x)\subset\subset\Omega$. Since $u\in L^\infty(\R^n)$ and $g$ is nondecreasing, for the far field we have
$$
\int_{\R^n\setminus B_{\varepsilon_0}(x)}\big|g\big(D^su(x,y)\big)\big|\,\frac{dy}{|x-y|^{n+s}}
\le g\left(\frac{2\|u\|_{L^\infty(\R^n)}}{\varepsilon_0^{\,s}}\right)\int_{|z|\ge\varepsilon_0}\frac{dz}{|z|^{n+s}}<\infty .
$$
For the singular part, a change of variables and the symmetry of the domain of integration give, for $0<\varepsilon<\varepsilon_0$,
$$
\int_{\varepsilon<|x-y|<\varepsilon_0}g\big(D^su(x,y)\big)\frac{dy}{|x-y|^{n+s}}
=\frac12\int_{\varepsilon<|z|<\varepsilon_0}\Big[g\big(D_+(z)\big)+g\big(D_-(z)\big)\Big]\frac{dz}{|z|^{n+s}},
$$
where $D_\pm(z):=\big(u(x)-u(x\pm z)\big)|z|^{-s}$. Since $u\in C^2(\Omega)$, a second order Taylor expansion yields
$$
D_\pm(z)=\mp\frac{\nabla u(x)\cdot z}{|z|^{s}}+h_\pm(z),\qquad |h_\pm(z)|\le C\,|z|^{2-s},
$$
with $C$ depending on $\|D^2u\|_{L^\infty(B_{\varepsilon_0}(x))}$. Writing $t:=\nabla u(x)\cdot z\,|z|^{-s}$ and using that $g$ is odd and locally Lipschitz, with Lipschitz constant $L$ on the relevant bounded interval, we obtain
$$
\big|g(D_+(z))+g(D_-(z))\big|=\big|g(t+h_-(z))-g(t-h_+(z))\big|\le L\big(|h_+(z)|+|h_-(z)|\big)\le 2LC\,|z|^{2-s} .
$$
Therefore the symmetrized integrand is dominated in absolute value by $LC\,|z|^{2-2s-n}$, which is integrable near the origin:
$$
LC\int_{|z|<\varepsilon_0}|z|^{2-s}\frac{dz}{|z|^{n+s}}=LC\int_{|z|<\varepsilon_0}\frac{dz}{|z|^{n+2s-2}}<\infty,
$$
since $2s-2<0$. Hence, by dominated convergence, the limit in \eqref{glap} exists and is finite.
\end{proof}

\begin{obs}
The local Lipschitz continuity of $g$ is the analogue, in this setting, of the restriction $p\ge2$ in the case $G(t)=t^p$: for $1<p<2$ the function $g(t)=|t|^{p-2}t$ is not Lipschitz near the origin. It is satisfied, for instance, by $G(t)=\frac{|t|^p}{p}+\frac{|t|^q}{q}$ and by $G(t)=|t|^p\log(1+|t|)$, with $p,q\ge2$.
\end{obs}
For the variational setting, the relevant definition of $(-\Delta_g)^s$ is the one obtained by duality, and no regularity of $u$ is needed. Recall that, by \eqref{Liebermann},
\begin{equation}\label{conjugada}
\bar G\big(g(t)\big)=t\,g(t)-G(t)\le (p^+-1)\,G(t),\qquad t\ge0,
\end{equation}
where the identity is the equality case in Young's inequality \eqref{young} and the bound follows from \eqref{Liebermann}; see also \cite[Lemma 2.3]{Salort}, where the estimate $\bar G(g(t))\le p^+G(t)$ is obtained. In particular, $g(D^su)\in L_{\bar G}(\nu_n)$ whenever $u\in W^{s,G}(\R^n)$. Consequently, by H\"older's inequality in Orlicz spaces, the quantity
\begin{equation}\label{debil}
\langle (-\Delta_g)^su,v\rangle:=\iint_{\R^{2n}}g\big(D^su\big)\,D^sv\,d\nu_n
\end{equation}
is finite for every $u,v\in W^{s,G}(\R^n)$, and
$$
\big|\langle (-\Delta_g)^su,v\rangle\big|\le 2\,\big\|g(D^su)\big\|_{\bar G,\nu_n}\,[v]_{s,G}.
$$
Thus $(-\Delta_g)^su$ defines a distribution; more precisely, since $W_A^{s,G}(\Omega)$ is a closed subspace of $W^{s,G}(\R^n)$, \eqref{debil} defines an operator $
(-\Delta_g)^s\colon W_A^{s,G}(\Omega)\to \big(W_A^{s,G}(\Omega)\big)'.$
Moreover, $\langle(-\Delta_g)^su,v\rangle$ is precisely the Gateaux derivative of the modular,
$$
\langle (-\Delta_g)^su,v\rangle=\frac{d}{d\varepsilon}\Phi_{s,G}(u+\varepsilon v)\Big|_{\varepsilon=0},
$$
where the differentiation under the integral sign is justified by \eqref{young}, \eqref{conjugada} and \eqref{LyDelta2}; this is the reason for the normalization chosen in \eqref{glap}. Indeed, if $u$ is as in Lemma \ref{lema:bienbdef} and $v\in C_c^\infty(\Omega)$, then $g(D^su(y,x))=-g(D^su(x,y))$ and the measure $\nu_n$ is symmetric, so that
$$
\iint_{\R^{2n}}g(D^su)\,D^sv\,d\nu_n=2\int_{\R^n}v(x)\left(\int_{\R^n}g\big(D^su(x,y)\big)\frac{dy}{|x-y|^{n+s}}\right)dx=\int_{\R^n}(-\Delta_g)^su\,v\,dx,
$$
and the two definitions agree.

\subsection{The eigenvalue problem with obstacle}

Throughout this subsection, $A\subset\Omega$ is a measurable set and $\mu>0$. The problem associated with the constant $\lambda_{s,\mu}(A)$ is
\begin{equation}\label{problema}
\begin{cases}
(-\Delta_g)^su=\lambda\, g(u) & \text{in } \Omega\setminus A,\\
u=0 & \text{in } A\cup(\R^n\setminus\Omega),\\
\Gcal(u)=\mu,
\end{cases}
\end{equation}
where $\lambda\in\R$ plays the role of an eigenvalue. Observe that, by \eqref{conjugada} applied to $u$ instead of $D^su$, we have $g(u)\in L_{\bar G}(\Omega)$ whenever $u\in L_G(\Omega)$, so that the right-hand side of \eqref{problema} is also well defined in the dual sense.

\begin{defi}\label{def:debil}
We say that a pair $(u,\lambda)\in W_A^{s,G}(\Omega)\times\R$ is a \emph{weak solution} of \eqref{problema} if $\Gcal(u)=\mu$ and
\begin{equation}\label{debilproblema}
\iint_{\R^{2n}}g\big(D^su\big)\,D^sv\,d\nu_n=\lambda\int_\Omega g(u)\,v\,dx\qquad\text{for every } v\in W_A^{s,G}(\Omega).
\end{equation}
\end{defi}

\begin{obs}
Since $W_A^{s,G}(\Omega)$ is a linear subspace, the test functions in \eqref{debilproblema} are allowed to have any sign, so that \eqref{debilproblema} is an equation and not a variational inequality. Note also that the condition $u=0$ in $A\cup(\R^n\setminus\Omega)$ is encoded in the requirement $u\in W_A^{s,G}(\Omega)$, and that no boundary datum has to be prescribed on $\partial\Omega$, the exterior condition playing that role.
\end{obs}
The following result ensures existence of extremals for $\lambda_{s,\mu}(A)$.
For $A=\emptyset$ this is \cite[Theorem 4.2]{Salort}, and for $G(t)=t^p$ it is the analogue of \cite[Theorem 2.11]{FBS}. We include the proof, which is the direct method, both for completeness and because in the presence of an obstacle one has to make sure that the class of admissible functions is nonempty.

\begin{teo}\label{teo:existeA}
Let $0<s<1$ and $\mu>0$. Let $\Omega\subset \R^n$ be open and bounded, and let $A\subsetneq\Omega$ be closed. If \eqref{Liebermann} is fulfilled then there exists an extremal $u\in W_A^{s,G}(\Omega)$ for $\lambda_{s,\mu}(A)$.
\end{teo}
\begin{proof}
Since $A$ is closed and $A\subsetneq\Omega$, the open set $\Omega\setminus A$ is nonempty and therefore contains a ball $B$; given $\psi\in C_c^\infty(B)$ with $\psi\not\equiv0$, the function $t\mapsto\Gcal(t\psi)$ is continuous, vanishes at $t=0$ and tends to $+\infty$ as $t\to\infty$, so that $\Gcal(t\psi)=\mu$ for some $t>0$ and $t\psi\in W_A^{s,G}(\Omega)$ is admissible. In particular $\lambda_{s,\mu}(A)<\infty$. Let then $\{u_k\}_{k\in \N}\subset W_A^{s,G}(\Omega)$ be a minimizing sequence for $\lambda_{s,\mu}(A)$, that is
$$
\Gcal(u_k)=\mu\quad \text{for every } k\in\N\qquad\text{and}\qquad \lim_{k\to \infty}\Phi_{s,G}(u_k)=\lambda_{s,\mu}(A).
$$
In particular
$$
\Phi_{s,G}(u_k)\leq C\qquad\forall k\in \N,
$$
where $C$ is independent of $k$, so that $\{u_k\}_{k\in\N}$ is bounded in $W_0^{s,G}(\Omega)$ by Remark \ref{obs:normamodular}. Since $W_0^{s,G}(\Omega)$ is reflexive and the embedding $W_0^{s,G}(\Omega)\hookrightarrow L_G(\Omega)$ is compact (Proposition \ref{prop:espacios}), there exist $u\in W_0^{s,G}(\Omega)$ and a subsequence, that we still denote by $\{u_k\}_{k\in\N}$, such that
\begin{align}
u_k&\rightharpoonup u\quad\text{weakly in } W_0^{s,G}(\Omega),\label{convdebil}
\\
u_k&\to u\quad\text{in }L_G(\Omega),\label{convLG}
\\
u_k&\to u\quad\text{a.e. in }\Omega.\label{convctp}
\end{align}
Since $u_k=0$ a.e. in $A$ for every $k$, \eqref{convctp} gives $u=0$ a.e. in $A$, and therefore $u\in W_A^{s,G}(\Omega)$. On the other hand, by \eqref{Liebermann} the functional $\Gcal$ is continuous on $L_G(\Omega)$, so that \eqref{convLG} yields
$$
\Gcal(u)=\lim_{k\to \infty}\Gcal(u_k)=\mu,
$$
and hence $u$ is admissible for $\lambda_{s,\mu}(A)$. Finally, $\Phi_{s,G}$ is convex and strongly lower semicontinuous, hence weakly lower semicontinuous, and \eqref{convdebil} gives
$$
\Phi_{s,G}(u)\leq \liminf_{k\to\infty}\Phi_{s,G}(u_k)=\lambda_{s,\mu}(A)\leq \Phi_{s,G}(u).
$$
Thus $\Phi_{s,G}(u)=\lambda_{s,\mu}(A)$, that is, $u$ is an extremal for $\lambda_{s,\mu}(A)$.
\end{proof}

The following result shows that every extremal of $\lambda_{s,\mu}(A)$ solves \eqref{problema} weakly, for a suitable eigenvalue $\lambda>0$. It is the analogue of \cite[Theorem 2.13]{FBS} and, in the presence of an obstacle, the counterpart of \cite[Theorem 4.2 and Corollary 5.3]{Salort}, where the case $A=\emptyset$ is treated. Since $W_A^{s,G}(\Omega)$ is a closed subspace of $W_0^{s,G}(\Omega)$, the proof is the same; we include it for completeness.

\begin{prop}\label{prop:EL}
Let $0<s<1$ and assume \eqref{Liebermann}. If $u\in W_A^{s,G}(\Omega)$ is an extremal for $\lambda_{s,\mu}(A)$, then there exists $\lambda>0$ such that the pair $(u,\lambda)$ is a weak solution of \eqref{problema} in the sense of Definition \ref{def:debil}. Moreover,
\begin{equation}\label{cotaslambda}
\frac{p^-}{p^+}\cdot \frac{\lambda_{s,\mu}(A)}{\mu}\ \le\ \lambda\ \le\ \frac{p^+}{p^-}\cdot\frac{\lambda_{s,\mu}(A)}{\mu}.
\end{equation}
\end{prop}
\begin{proof}
By \eqref{Liebermann} the functionals
$$
\Phi_{s,G}\colon W^{s,G}(\R^n)\to\R,\qquad \Gcal\colon L_G(\Omega)\to\R
$$
are of class $C^1$, with
$$
\langle \Phi_{s,G}'(u),v\rangle=\iint_{\R^{2n}}g(D^su)\,D^sv\,d\nu_n,\qquad
\langle \Gcal'(u),v\rangle=\int_\Omega g(u)\,v\,dx ,
$$
both expressions being finite by \eqref{conjugada} and H\"older's inequality in Orlicz spaces.

The function $u$ minimizes $\Phi_{s,G}$ on the set $\{v\in W_A^{s,G}(\Omega)\colon \Gcal(v)=\mu\}$. Moreover, $u$ is a regular point of the constraint: by \eqref{Liebermann},
$$
\langle \Gcal'(u),u\rangle=\int_\Omega g(u)u\,dx\ \ge\ p^-\int_\Omega G(u)\,dx=p^-\mu>0,
$$
so that $\Gcal'(u)\neq0$ in $\big(W_A^{s,G}(\Omega)\big)'$. Since $W_A^{s,G}(\Omega)$ is a Banach space, the Lagrange multiplier rule yields the existence of $\lambda\in\R$ such that $\Phi_{s,G}'(u)=\lambda\,\Gcal'(u)$ in $\big(W_A^{s,G}(\Omega)\big)'$, which is precisely \eqref{debilproblema}. Together with $u\in W_A^{s,G}(\Omega)$ and $\Gcal(u)=\mu$, this says that $u$ is a weak solution of \eqref{problema}.

Finally, taking $v=u$ in \eqref{debilproblema} we get
$$
\iint_{\R^{2n}}g(D^su)\,D^su\,d\nu_n=\lambda\int_\Omega g(u)u\,dx ,
$$
and \eqref{Liebermann}, together with $\Phi_{s,G}(u)=\lambda_{s,\mu}(A)$ and $\Gcal(u)=\mu$, bounds both sides:
$$
p^-\lambda_{s,\mu}(A)\le \iint_{\R^{2n}}g(D^su)\,D^su\,d\nu_n\le p^+\lambda_{s,\mu}(A),
\qquad
0<p^-\mu\le \int_\Omega g(u)u\,dx\le p^+\mu .
$$
Since $\lambda_{s,\mu}(A)>0$ by Remark \ref{obs:funcional}(2), the left--hand side of the identity is positive, and so is $\lambda$; dividing the two chains of inequalities gives \eqref{cotaslambda}.
\end{proof}

\begin{obs}
In contrast with the homogeneous case, the eigenvalue $\lambda$ does not coincide with the constant $\lambda_{s,\mu}(A)$. For $G(t)=|t|^p$ one has $g(t)t=pG(t)$, and \eqref{cotaslambda} becomes the identity
$$
\lambda=\frac{\lambda_{s,\mu}(A)}{\mu},
$$
so that the two constants agree only for $\mu=1$; in that case \eqref{problema} reduces to the eigenvalue equation considered in \cite[Theorem 2.13]{FBS}.
\end{obs}
\subsection{The eigenvalue problem with a potential}
Throughout this subsection, $\sigma>0$ and $\varphi\in L^\infty(\Omega)$ is nonnegative. The problem associated with the constant $\lambda_{s,\mu}(\varphi,\sigma)$ is
\begin{equation}\label{problemapot}
\begin{cases}
(-\Delta_g)^su+\sigma\,g(u)\,\varphi=\lambda\,g(u) & \text{in } \Omega,\\
u=0 & \text{in } \R^n\setminus\Omega,\\
\Gcal(u)=\mu,
\end{cases}
\end{equation}
where $\lambda\in\R$ plays again the role of an eigenvalue. Equivalently, the first equation in \eqref{problemapot} can be written as
$$
(-\Delta_g)^su=(\lambda-\sigma\varphi)\,g(u),
$$
that is, as an eigenvalue problem with the weight $\lambda-\sigma\varphi$; written in this form it is apparent that the penalization pushes the extremals away from the region where $\varphi$ is large.

\begin{defi}\label{def:debilpot}
We say that a pair $(u,\lambda)\in W_0^{s,G}(\Omega)\times\R$ is a \emph{weak solution} of \eqref{problemapot} if $\Gcal(u)=\mu$ and
\begin{equation}\label{debilproblemapot}
\iint_{\R^{2n}}g\big(D^su\big)\,D^sv\,d\nu_n+\sigma\int_\Omega g(u)\,v\,\varphi\,dx=\lambda\int_\Omega g(u)\,v\,dx\qquad\text{for every } v\in W_0^{s,G}(\Omega).
\end{equation}
\end{defi}

All the terms in \eqref{debilproblemapot} are finite: for the first one this was observed in Subsection \ref{subsec:operador}, and for the remaining two it suffices to recall that $g(u)\in L_{\bar G}(\Omega)$ by \eqref{conjugada} and that $\varphi\in L^\infty(\Omega)$. Observe also that, in contrast with \eqref{debilproblema}, the test functions here range over the whole of $W_0^{s,G}(\Omega)$, so that no exceptional set appears.

\begin{obs}\label{obs:existePot}
The existence of an extremal for the constant $\lambda_{s,\mu}(\varphi,\sigma)$ follows with the same argument as in the proof of Theorem \ref{teo:existeA}. The only difference is the additional term, which is continuous with respect to the strong convergence in $L_G(\Omega)$ because $\varphi\in L^\infty(\Omega)$; let us also point out that here the class of admissible functions does not depend on $\varphi$, so that no further hypothesis is needed to guarantee that it is nonempty.
\end{obs}

\begin{teo}\label{teo:existePot}
Let $0<s<1$ and $\mu>0$, let $\Omega\subset\R^n$ be open and bounded, let $\sigma>0$ and let $\varphi\in L^\infty(\Omega)$ be nonnegative. If \eqref{Liebermann} is fulfilled then there exists an extremal $u\in W_0^{s,G}(\Omega)$ for $\lambda_{s,\mu}(\varphi,\sigma)$.
\end{teo}

As in the case of the obstacle, the extremals solve the associated problem weakly.

\begin{prop}\label{prop:ELpot}
Let $0<s<1$ and assume \eqref{Liebermann}. If $u\in W_0^{s,G}(\Omega)$ is an extremal for $\lambda_{s,\mu}(\varphi,\sigma)$, then there exists $\lambda>0$ such that the pair $(u,\lambda)$ is a weak solution of \eqref{problemapot} in the sense of Definition \ref{def:debilpot}. Moreover,
\begin{equation}\label{cotaslambdapot}
\frac{p^-}{p^+}\cdot \frac{\lambda_{s,\mu}(\varphi,\sigma)}{\mu}\ \le\ \lambda\ \le\ \frac{p^+}{p^-}\cdot\frac{\lambda_{s,\mu}(\varphi,\sigma)}{\mu}.
\end{equation}
\end{prop}

\begin{proof}
The proof is that of Proposition \ref{prop:EL}. Indeed, by \eqref{Liebermann} the functional $u\mapsto\int_\Omega G(u)\varphi\,dx$ is of class $C^1$ on $L_G(\Omega)$, with derivative $v\mapsto\int_\Omega g(u)\,v\,\varphi\,dx$, since $\varphi\in L^\infty(\Omega)$; and \eqref{cotaslambdapot} follows by testing \eqref{debilproblemapot} with $v=u$ and using \eqref{Liebermann} in the three resulting terms, which gives
$$
p^-\lambda_{s,\mu}(\varphi,\sigma)\le \lambda\int_\Omega g(u)\,u\,dx\le p^+\lambda_{s,\mu}(\varphi,\sigma),\qquad p^-\mu\le \int_\Omega g(u)\,u\,dx\le p^+\mu .
$$
\end{proof}

\section{The optimal design problems}\label{sec:optimal}
Throughout this section we assume that $\Omega\subset\R^n$ is open and bounded and that $G$ satisfies \eqref{Liebermann}. In particular, the embedding $W_0^{s,G}(\Omega)\hookrightarrow L_G(\Omega)$ is compact and modular and norm convergence are equivalent.

\subsection{Optimization for the $(s,G)$--hard obstacle problem}\label{subsec:hard}

We begin this subsection by showing that the extremals of $\lambda_{s,\mu}(A)$ have constant sign, which allows us to deal only with nonnegative extremals in what follows. We then give a characterization of the constant $\Lambda_{s,\mu}(\alpha)$, prove the existence of an optimal configuration, and finally show that any optimal set coincides, up to a null set, with the set where the corresponding extremal vanishes.

The argument in the next lemma is the same as in \cite[Lemma 3.1]{FBS}; we include it for completeness, since it does not use \eqref{Liebermann}.
\begin{lema}\label{lema:signo}
Let $0<s<1$, let $\Omega\subset\R^n$ be open and bounded, let $A\subset\Omega$ be a measurable set with $|A|>0$ and let $\mu>0$. If $u\in W_A^{s,G}(\Omega)$ is an extremal for $\lambda_{s,\mu}(A)$, then $|u|$ is also an extremal and either $u\geq 0$ a.e. in $\Omega$ or $u\leq 0$ a.e. in $\Omega$.
\end{lema}
\begin{proof}
Since $\big||a|-|b|\big|\le|a-b|$ and $G$ is even and nondecreasing on $[0,\infty)$, we have $G(D^s|u|)\le G(D^su)$ for a.e. $(x,y)\in\R^{2n}\setminus\Delta$. Hence $|u|\in W_A^{s,G}(\Omega)$ and, $G$ being even, $\Gcal(|u|)=\Gcal(u)=\mu$, and
$$
\iint_{\R^{2n}}G(D^s|u|)\,d\nu_n\le\iint_{\R^{2n}}G(D^su)\,d\nu_n=\lambda_{s,\mu}(A)<\infty,
$$
so $|u|$ is an extremal. Assume now that $P=\{u>0\}$ and $N=\{u<0\}$ both have positive measure; observe that $P,N\subset\Omega$, since $u=0$ a.e. in $\R^n\setminus\Omega$. For $(x,y)\in P\times N$ we have $\big||u(x)|-|u(y)|\big|<|u(x)-u(y)|$, and since $g>0$ on $(0,\infty)$, $G$ is strictly increasing on $[0,\infty)$, so the pointwise inequality above is strict on $P\times N$. As $\nu_n(P\times N)\ge \operatorname{diam}(\Omega)^{-n}|P||N|>0$ and the right-hand side is finite, we get
$$
\iint_{\R^{2n}}G(D^s|u|)\,d\nu_n<\lambda_{s,\mu}(A),
$$
a contradiction.
\end{proof}
The following lemma characterizes the $(s,G)$--hard obstacle problem, and is the counterpart of \cite[Lemma 3.2]{FBS}.
\begin{lema}\label{lema:caract}
Let $\alpha\in (0,1)$ and $\mu>0$. Then
\begin{equation}\label{caractLambda_s}
\Lambda_{s,\mu}(\alpha)=\inf \left \{\iint_{\R^{2n}}G(D^s u)\,d\nu_n\colon u\in W_0^{s,G}(\Omega),\int_\Omega G(u)\,dx=\mu,|\{u=0\}\cap\Omega|\geq\alpha|\Omega| \right \}.
\end{equation}
\end{lema}
\begin{proof}
Let us denote by $\tilde\Lambda_{s,\mu}(\alpha)$ the right-hand side of \eqref{caractLambda_s}.

Let $A\in\mathcal{A}_\alpha$ and let $\varepsilon>0$. We may assume that $\lambda_{s,\mu}(A)<\infty$, since otherwise the inequality below is trivial. By the definition of $\lambda_{s,\mu}(A)$ there exists $u\in W_A^{s,G}(\Omega)$ with $\Gcal(u)=\mu$ and
$$
\iint_{\R^{2n}}G(D^su)\,d\nu_n\le \lambda_{s,\mu}(A)+\varepsilon .
$$
Since $u=0$ a.e. in $A$, we have $|\{u=0\}\cap\Omega|\ge|A|=\alpha|\Omega|$, so that $u$ is admissible in \eqref{caractLambda_s} and therefore $\tilde\Lambda_{s,\mu}(\alpha)\le\lambda_{s,\mu}(A)+\varepsilon$. Letting $\varepsilon\to 0$ and taking the infimum over $A\in\mathcal{A}_\alpha$ we obtain
\begin{equation}\label{caract1}
\tilde\Lambda_{s,\mu}(\alpha)\le \Lambda_{s,\mu}(\alpha).
\end{equation}
Conversely, let $v\in W_0^{s,G}(\Omega)$ be such that $\Gcal(v)=\mu$ and $|\{v=0\}\cap\Omega|\ge\alpha|\Omega|$, and write $E:=\{v=0\}\cap\Omega$. The function $f(t):=|E\cap B_t(0)|$ is nondecreasing and continuous, $f(0)=0$, and $f(t)=|E|\ge\alpha|\Omega|$ for $t$ large enough, since $\Omega$ is bounded. Hence, by the intermediate value theorem, there exists $t>0$ such that the set $A:=E\cap B_t(0)$ satisfies $|A|=\alpha|\Omega|$, that is, $A\in\mathcal{A}_\alpha$. Since $v=0$ a.e. in $A$, we have $v\in W_A^{s,G}(\Omega)$, and $\Gcal(v)=\mu$, so that
$$
\Lambda_{s,\mu}(\alpha)\le\lambda_{s,\mu}(A)\le \iint_{\R^{2n}}G(D^sv)\,d\nu_n .
$$
Taking the infimum over all such $v$ we get $\Lambda_{s,\mu}(\alpha)\le\tilde\Lambda_{s,\mu}(\alpha)$, which together with \eqref{caract1} proves the lemma.
\end{proof}
\begin{obs}
Observe that the above argument does not use any homogeneity of $G$: the constraint $\Gcal(u)=\mu$ is preserved by both inequalities, and no rescaling of the admissible functions is needed. This is in contrast with the case $G(t)=t^p$, where the corresponding characterization, see \cite[Lemma 3.2]{FBS}, is obtained after normalizing $\|u\|_{L^p(\Omega)}=1$.
\end{obs}
We have the following existence of optimal configuration result, whose proof is analogous to that of \cite[Theorem 3.3]{FBS}.

\begin{teo}\label{exop}
Let $\alpha\in (0,1)$ and $\mu>0$. Then there exist
\begin{enumerate}
\item a measurable set $A \subset \Omega$ such that $|A|=\alpha |\Omega|$, and
$$
\Lambda_{s,\mu}(\alpha)=\lambda_{s,\mu}(A).
$$
\item a function $u\in W_0^{s,G}(\Omega)$ with $|\{u=0\}\cap \Omega|\geq \alpha |\Omega|$ and $\Gcal(u)=\mu$ such that
$$
\Lambda_{s,\mu}(\alpha)=\iint_{\R^{2n}}G(D^s u)\,d\nu_n.
$$
\end{enumerate}
\end{teo}
\begin{proof}
We first observe that (2) implies (1). Indeed, let $u$ be as in (2). Arguing as in the proof of Lemma \ref{lema:caract}, there exists a measurable set $A\subset \{u=0\}\cap \Omega$ such that $|A|=\alpha |\Omega|$, and hence
$$
\Lambda_{s,\mu}(\alpha)\leq \lambda_{s,\mu}(A)\leq \iint_{\R^{2n}}G(D^s u)\,d\nu_n=\Lambda_{s,\mu}(\alpha),
$$
so that $\Lambda_{s,\mu}(\alpha)=\lambda_{s,\mu}(A)$. Therefore it suffices to prove (2).

Let $\{u_i\}_{i\in \N}$ be a minimizing sequence for $\Lambda_{s,\mu}(\alpha)$. By Lemma \ref{lema:caract} we may assume that $\{u_i\}_{i\in\N}\subset W_0^{s,G}(\Omega)$, $\Gcal(u_i)=\mu$ and $|\{u_i=0\}\cap\Omega|\geq \alpha|\Omega|$ for every $i\in\N$, and
$$
\Lambda_{s,\mu}(\alpha)=\lim_{i\to\infty} \iint_{\R^{2n}}G(D^s u_i)\,d\nu_n,
$$
therefore
$$
\iint_{\R^{2n}}G(D^s u_i)\,d\nu_n\leq C\qquad \forall i\in \N,
$$
where $C>0$ does not depend on $i$, and so, by Remark \ref{obs:normamodular},
$$
\|u_i\|_{s,G}\leq \max\{1,\mu\}+\max\{1,C\}\qquad i\in\N,
$$
therefore, by the reflexivity of the space $W_0^{s,G}(\Omega)$, see Proposition \ref{prop:espacios}, there exists a function $u$ and a subsequence, that we still call $\{u_i\}_{i\in\N}$, such that $u_i\rightharpoonup u$ weakly in $W^{s,G}(\R^n)$; since $W_0^{s,G}(\Omega)$ is a closed, and hence weakly closed, subspace, we have $u\in W_0^{s,G}(\Omega)$. Moreover, by the compactness of the embedding $W_0^{s,G}(\Omega)\hookrightarrow L_G(\Omega)$, we may also assume that $u_i\to u$ in $L_G(\Omega)$ and $u_i\to u$ a.e. in $\Omega$ as $i\to\infty$. Then
$$
\int_\Omega G(u)\,dx=\lim_{i\to \infty}\int_\Omega G(u_i)\,dx=\mu,
$$
where we have used that convergence in $L_G(\Omega)$ implies convergence of the modulars, see Lemma \ref{lema:convG}. Moreover, since $u_i\to u$ a.e. in $\Omega$ we have $\{u\neq 0\}\subset \liminf_{i\to\infty}\{u_i\neq 0\}$ up to a null set, so that Fatou's lemma applied to $\chi_{\{u_i\neq 0\}}$ gives $|\{u\neq 0\}\cap\Omega|\leq \liminf_{i\to\infty}|\{u_i\neq 0\}\cap\Omega|$, that is, $\limsup_{i\to \infty} |\{u_i=0\}\cap\Omega|\leq |\{u=0\}\cap\Omega|$. Hence $|\{u=0\}\cap\Omega|\geq \alpha |\Omega|$ and so by Lemma \ref{lema:caract}
$$
\Lambda_{s,\mu}(\alpha)\leq \iint_{\R^{2n}}G(D^s u)\,d\nu_n.
$$
Finally, since all these functions vanish outside $\Omega$, we have $u_i\to u$ a.e. in $\R^n$ and therefore $D^su_i\to D^su$ a.e. in $\R^{2n}\setminus\Delta$, so that Fatou's lemma gives
$$
\iint_{\R^{2n}}G(D^s u)\,d\nu_n=\Phi_{s,G}(u)\leq \liminf_{i\to \infty}\Phi_{s,G}(u_i)=\Lambda_{s,\mu}(\alpha).
$$
Both inequalities together give $\Phi_{s,G}(u)=\Lambda_{s,\mu}(\alpha)$, that is, $u$ is an extremal for $\Lambda_{s,\mu}(\alpha)$.
\end{proof}
We now present the main result of this section: the optimal set given by Theorem \ref{exop} coincides, up to a null set, with the set where the corresponding extremal vanishes. We first show that the zero set of an extremal has exactly the prescribed measure, and then deduce the identification in Corollary \ref{coro:identif}.
\begin{teo}\label{teo:identif}
With the same hypotheses as in Theorem \ref{exop}, let $u\in W_0^{s,G}(\Omega)$ be an extremal for $\Lambda_{s,\mu}(\alpha)$. Then
$$
|\{u=0\}\cap\Omega|=\alpha|\Omega|.
$$
\end{teo}
\begin{proof}
By Lemma \ref{lema:signo} we may assume that $u\ge 0$. Suppose, arguing by contradiction, that $|\{u=0\}\cap\Omega|>\alpha|\Omega|$, and set
$$
\theta:=G^{-1}\left(\frac{\mu}{2|\Omega|}\right),\qquad F:=\{u\ge\theta\},
$$
so that $|F|>0$. Indeed, otherwise $\mu=\int_\Omega G(u)\,dx\le G(\theta)|\Omega|=\mu/2$, a contradiction.

Let $x_0$ be a point of density $1$ of $\{u=0\}\cap\Omega$, which exists by the Lebesgue density theorem, and let $r>0$ be small enough so that $B_r(x_0)\subset\subset\Omega$ and
$$
\big|\{u=0\}\cap B_{r/2}(x_0)\big|\ge\tfrac12\big|B_{r/2}(x_0)\big|,\qquad
|F\setminus B_r(x_0)|>0,\qquad
\big|B_r(x_0)\big|<\big|\{u=0\}\cap\Omega\big|-\alpha|\Omega| .
$$
We define
$$
A:=\big(\{u=0\}\cap\Omega\big)\setminus B_r(x_0),
$$
so that $|A|>\alpha|\Omega|$ by the last condition above.

Let $\psi\in C_c^\infty(B_r(x_0))$ be a cut-off function such that $0\le\psi\le1$ and $\psi\equiv1$ in $B_{r/2}(x_0)$, and for $0<\delta<\theta/2$ consider
$$
\varphi_\delta:=(\delta-u)_+\,\psi .
$$
Then $0\le\varphi_\delta\le\delta$ and $\supp \varphi_\delta\subset B_r(x_0)$, so that $\varphi_\delta$ vanishes a.e. in $A$ and in $\R^n\setminus\Omega$; hence $\varphi_\delta\in W_A^{s,G}(\Omega)$.

We now claim that $u$ is an extremal for $\lambda_{s,\mu}(A)$. On the one hand, since $|A|>\alpha|\Omega|$, every function vanishing a.e. in $A$ is admissible in the characterization \eqref{caractLambda_s}, whence $\Lambda_{s,\mu}(\alpha)\le\lambda_{s,\mu}(A)$. On the other hand, $u\in W_A^{s,G}(\Omega)$ and $\Gcal(u)=\mu$, so that
$$
\lambda_{s,\mu}(A)\le \iint_{\R^{2n}}G(D^su)\,d\nu_n=\Lambda_{s,\mu}(\alpha),
$$
and the claim follows.

Therefore, by Proposition \ref{prop:EL}, there exists $\lambda>0$ such that
\begin{equation}\label{EL}
\iint_{\R^{2n}}g(D^su)\,D^sv\,d\nu_n=\lambda\int_\Omega g(u)\,v\,dx\qquad\text{for every } v\in W_A^{s,G}(\Omega).
\end{equation}
We now test \eqref{EL} with $v=\varphi_\delta$ and write
$$
I_\delta:=\iint_{\R^{2n}}g(D^su)\,D^s\varphi_\delta\,d\nu_n=\lambda\int_\Omega g(u)\,\varphi_\delta\,dx=:J_\delta .
$$
Since $\varphi_\delta\ge0$ and $u\ge 0$, we have $J_\delta\ge 0$. Moreover, $\varphi_\delta$ vanishes outside the set $\{u<\delta\}\cap B_r(x_0)$ and $\varphi_\delta\le\delta$ there, so that, $g$ being nondecreasing,
\begin{equation}\label{RHS}
0\ \le\ J_\delta\ \le\ \lambda\,g(\delta)\,\delta\,\big|B_r(x_0)\big| .
\end{equation}
In particular $J_\delta=o(\delta)$ as $\delta\to0^+$, since $g(\delta)\to0$.

We now estimate $I_\delta$ from above. The integrand
$$
\mathcal{I}:=g\big(D^su\big)\,D^s\varphi_\delta
$$
is symmetric in $(x,y)$, because $g$ is odd and both $D^su$ and $D^s\varphi_\delta$ are antisymmetric, so that in discussing its sign we may assume $u(x)<u(y)$; note also that $\mathcal{I}=0$ when $u(x)=u(y)$. Under that assumption $g(D^su)<0$, and hence $\mathcal{I}\le0$ if and only if $\varphi_\delta(x)\ge\varphi_\delta(y)$. This holds trivially when $u(y)\ge\delta$, since then $\varphi_\delta(y)=0$. When $u(y)<\delta$ we write
\begin{equation}\label{clave}
\varphi_\delta(y)-\varphi_\delta(x)=\big(\delta-u(y)\big)\big(\psi(y)-\psi(x)\big)+\big(u(x)-u(y)\big)\psi(x)\le\big(\delta-u(y)\big)\big(\psi(y)-\psi(x)\big),
\end{equation}
the inequality because $u(x)<u(y)$ and $\psi\ge0$. If $\psi(y)\le\psi(x)$, the right-hand side of \eqref{clave} is nonpositive and again $\mathcal{I}\le0$; if $\psi(y)>\psi(x)$, then $0<\delta-u(y)<\delta$ gives $\varphi_\delta(y)-\varphi_\delta(x)\le\delta\,|\psi(x)-\psi(y)|$. Since in this last case $u(x)<u(y)<\delta$ forces $u(y)>0$, setting
$$
\mathcal{B}_\delta:=\{u<\delta\}\times\{u<\delta\}
$$
we have proved that
\begin{equation}\label{signo}
\mathcal{I}\le 0\quad\text{in }\R^{2n}\setminus\mathcal{B}_\delta,\qquad
\mathcal{I}\le \delta\,\big|g(D^su)\big|\,\big|D^s\psi\big|\quad\text{in }\mathcal{B}_\delta .
\end{equation}
Let now $d:=\operatorname{diam}(\Omega)$, let
$$
S:=\big(\{u=0\}\cap B_{r/2}(x_0)\big)\times\big(F\setminus B_r(x_0)\big)\ \subset\ \R^{2n},
$$
and split
$$
I_\delta=\iint_{\R^{2n}\setminus S}\mathcal{I}\,d\nu_n+\iint_{S}\mathcal{I}\,d\nu_n .
$$
If $(x,y)\in S$ then $u(x)=0$ and $\psi(x)=1$, so that $\varphi_\delta(x)=\delta$, while $u(y)\ge\theta>\delta$ gives $\varphi_\delta(y)=0$; moreover $|x-y|\le d$. Since $g$ is odd and nondecreasing,
$$
\mathcal{I}(x,y)=-g\left(\frac{u(y)}{|x-y|^{s}}\right)\frac{\delta}{|x-y|^{s}}\ \le\ -\delta\,g\big(\theta d^{-s}\big)\,d^{-s},
$$
and therefore, using $d\nu_n\ge d^{-n}dx\,dy$,
\begin{equation}\label{gain}
\iint_{S}\mathcal{I}\,d\nu_n\ \le\ -K\delta,\qquad
K:=g\big(\theta d^{-s}\big)\,d^{-s-n}\,\big|\{u=0\}\cap B_{r/2}(x_0)\big|\,\big|F\setminus B_r(x_0)\big| ,
\end{equation}
and $K>0$ by the choice of $r$.

The same computation shows that $S\cap\mathcal{B}_\delta=\emptyset$, since $u=0$ on the first factor of $S$ and $u\ge\theta>\delta$ on the second one. Hence $\mathcal{B}_\delta\subset\R^{2n}\setminus S$, and since $\mathcal{I}\le0$ on $(\R^{2n}\setminus S)\setminus\mathcal{B}_\delta$ by \eqref{signo}, the second bound in \eqref{signo} gives
\begin{equation}\label{error}
\iint_{\R^{2n}\setminus S}\mathcal{I}\,d\nu_n\ \le\ \iint_{\mathcal{B}_\delta}\mathcal{I}\,d\nu_n\ \le\ \delta\iint_{\mathcal{B}_\delta}\big|g(D^su)\big|\,\big|D^s\psi\big|\,d\nu_n=:\delta\,\varepsilon(\delta).
\end{equation}
Finally, $\varepsilon(\delta)\to0$ as $\delta\to0^+$. Indeed, the integrand belongs to $L^1(\nu_n)$, because $g(D^su)\in L_{\bar G}(\nu_n)$ by \eqref{conjugada}, $\psi\in W^{s,G}(\R^n)$ and H\"older's inequality in Orlicz spaces applies; and the sets $\mathcal{B}_\delta$ decrease, as $\delta\to0^+$, to $\{u=0\}\times\{u=0\}$, where $D^su=0$ and hence $g(D^su)=0$.

Adding \eqref{gain} and \eqref{error} and recalling \eqref{RHS} we obtain
$$
0\ \le\ J_\delta=I_\delta\ \le\ \delta\big(\varepsilon(\delta)-K\big)\ <\ 0
$$
as soon as $\delta$ is small enough that $\varepsilon(\delta)<K$, a contradiction. Since the inequality $|\{u=0\}\cap\Omega|\ge\alpha|\Omega|$ holds because $u$ is admissible in \eqref{caractLambda_s}, the proof is complete.
\end{proof}

\begin{obs}
In the case $G(t)=t^p$ the corresponding statement, \cite[Theorem 3.4]{FBS}, is obtained from a strong minimum principle, \cite[Theorem 2.8]{FBS}, which guarantees that a nonnegative extremal is positive almost everywhere in $\Omega\setminus A$ and which in turn rests on \cite[Theorem A.1]{BF}. The proof given above uses no minimum principle, and no regularity theory for $(-\Delta_g)^s$ either: its only ingredients are the Euler--Lagrange equation \eqref{EL}, the sign of the integrand and the fact that the layer $\{0<u<\delta\}$ shrinks to a null set as $\delta\to0^+$. In particular, \eqref{Liebermann} is used only through \eqref{EL}. A minimum principle is of course also available in the present setting, as a consequence of the regularity theory for nonlocal problems with Orlicz growth (see for instance \cite{KLL}), and would provide an alternative proof.
\end{obs}

\begin{cor}\label{coro:identif}
Let $\alpha\in(0,1)$ and $\mu>0$, let $A\subset\Omega$ be an optimal set for $\Lambda_{s,\mu}(\alpha)$ and let $u\in W_A^{s,G}(\Omega)$ be an extremal for $\lambda_{s,\mu}(A)$. Then
$$
A=\{u=0\}\cap\Omega\qquad\text{up to a set of zero measure.}
$$
\end{cor}
\begin{proof}
Since $u\in W_A^{s,G}(\Omega)$ we have $u=0$ a.e. in $A$, that is, $A\subset\{u=0\}\cap\Omega$ up to a null set; in particular $|\{u=0\}\cap\Omega|\ge|A|=\alpha|\Omega|$. Moreover $\Phi_{s,G}(u)=\lambda_{s,\mu}(A)=\Lambda_{s,\mu}(\alpha)$ and $\Gcal(u)=\mu$, so that $u$ is admissible in \eqref{caractLambda_s} and is therefore an extremal for $\Lambda_{s,\mu}(\alpha)$. By Theorem \ref{teo:identif},
$$
\big|\{u=0\}\cap\Omega\big|=\alpha|\Omega|=|A| .
$$
A measurable set contained in another one of the same finite measure differs from it by a null set, which gives the conclusion.
\end{proof}

\subsection{Optimization for the $(s,G)$--soft obstacle problem}\label{subsec:soft}

We now turn to \eqref{Lambdasoft}. A potential $\varphi\in\mathcal{B}_\alpha$ at which the infimum in \eqref{Lambdasoft} is attained is called an \emph{optimal potential}, and if moreover $u$ is an extremal for $\lambda_{s,\mu}(\varphi,\sigma)$, the pair $(u,\varphi)$ is called an \emph{optimal pair}. We show that an optimal pair exists and that the optimal potential can always be taken to be a characteristic function, so that nothing is lost by relaxing the class $\mathcal{A}_\alpha$ to $\mathcal{B}_\alpha$.
\begin{teo}\label{teo:paroptimo}
Let $0<s<1$, $\mu>0$, $\sigma>0$ and $\alpha\in(0,1)$. Then there exists an optimal pair $(u,\varphi)\in W_0^{s,G}(\Omega)\times\mathcal{B}_\alpha$ for $\Lambda_{s,\mu}(\alpha,\sigma)$, with $u\ge0$. Moreover one can take $\varphi=\chi_D$ with
$$
\{u<\tau\}\cap\Omega\subset D\subset\{u\le\tau\}\cap\Omega
$$
for some $\tau\ge0$.
\end{teo}
\begin{proof}
Let $\{\varphi_k\}_{k\in\N}\subset\mathcal{B}_\alpha$ be a minimizing sequence for $\Lambda_{s,\mu}(\alpha,\sigma)$. Since $\{\varphi_k\}_{k\in\N}$ is bounded in $L^\infty(\Omega)$, which is the dual of the separable space $L^1(\Omega)$, there exist $\varphi\in L^\infty(\Omega)$ and a subsequence, still denoted by $\{\varphi_k\}_{k\in\N}$, such that
\begin{equation}\label{convdebilestrella}
\varphi_k\stackrel{*}{\rightharpoonup}\varphi\quad\text{in }L^\infty(\Omega).
\end{equation}
Testing \eqref{convdebilestrella} with the functions $\chi_E$, $E\subset\Omega$ measurable, and with $1$, we get $0\le\varphi\le1$ a.e. and $\int_\Omega\varphi\,dx=\alpha|\Omega|$, that is, $\varphi\in\mathcal{B}_\alpha$.

For every $k$ let $u_k\in W_0^{s,G}(\Omega)$ be an extremal for $\lambda_{s,\mu}(\varphi_k,\sigma)$, which exists by Theorem \ref{teo:existePot}; replacing $u_k$ by $|u_k|$, which does not increase $\Phi_{s,G}$ and leaves both $\Gcal(u_k)$ and $\int_\Omega G(u_k)\varphi_k\,dx$ unchanged, we may assume $u_k\ge0$. Since
$$
\Phi_{s,G}(u_k)\le \Phi_{s,G}(u_k)+\sigma\int_\Omega G(u_k)\varphi_k\,dx=\lambda_{s,\mu}(\varphi_k,\sigma)\longrightarrow\Lambda_{s,\mu}(\alpha,\sigma),
$$
the modulars $\Phi_{s,G}(u_k)$ are bounded, so that $\{u_k\}_{k\in\N}$ is bounded in $W_0^{s,G}(\Omega)$ by Remark \ref{obs:normamodular}. Arguing as in the proof of Theorem \ref{teo:existeA} we obtain $u\in W_0^{s,G}(\Omega)$, $u\ge0$, and a subsequence such that $u_k\rightharpoonup u$ weakly in $W_0^{s,G}(\Omega)$, $u_k\to u$ in $L_G(\Omega)$ and a.e. in $\Omega$; in particular $\Gcal(u)=\mu$, so that $u$ is admissible in \eqref{lambdasoft}.

We claim that
\begin{equation}\label{convcruzado}
\int_\Omega G(u_k)\,\varphi_k\,dx\longrightarrow\int_\Omega G(u)\,\varphi\,dx .
\end{equation}
Indeed,
$$
\int_\Omega G(u_k)\varphi_k\,dx-\int_\Omega G(u)\varphi\,dx=\int_\Omega\big(G(u_k)-G(u)\big)\varphi_k\,dx+\int_\Omega G(u)\big(\varphi_k-\varphi\big)dx .
$$
The first term is bounded in absolute value by $\|G(u_k)-G(u)\|_{L^1(\Omega)}$, which tends to $0$ by Lemma \ref{lema:convG}, and the second one tends to $0$ by \eqref{convdebilestrella}, since $G(u)\in L^1(\Omega)$.

Since $\Phi_{s,G}$ is convex and strongly lower semicontinuous, hence weakly lower semicontinuous, \eqref{convcruzado} gives
$$
\Phi_{s,G}(u)+\sigma\int_\Omega G(u)\varphi\,dx\le \liminf_{k\to\infty}\left(\Phi_{s,G}(u_k)+\sigma\int_\Omega G(u_k)\varphi_k\,dx\right)=\Lambda_{s,\mu}(\alpha,\sigma).
$$
On the other hand $u$ is admissible in \eqref{lambdasoft} and $\varphi\in\mathcal{B}_\alpha$, so that
$$
\Lambda_{s,\mu}(\alpha,\sigma)\le\lambda_{s,\mu}(\varphi,\sigma)\le \Phi_{s,G}(u)+\sigma\int_\Omega G(u)\varphi\,dx\le\Lambda_{s,\mu}(\alpha,\sigma),
$$
therefore $(u,\varphi)$ is an optimal pair.

It remains to show that the optimal potential can be taken to be a characteristic function. If $\psi\in\mathcal{B}_\alpha$ then, $u$ being admissible in \eqref{lambdasoft},
$$
\Phi_{s,G}(u)+\sigma\int_\Omega G(u)\varphi\,dx=\Lambda_{s,\mu}(\alpha,\sigma)\le\lambda_{s,\mu}(\psi,\sigma)\le\Phi_{s,G}(u)+\sigma\int_\Omega G(u)\psi\,dx,
$$
so that $\varphi$ minimizes the linear functional $\psi\mapsto\int_\Omega G(u)\psi\,dx$ on $\mathcal{B}_\alpha$. By the bathtub principle \cite[Theorem 1.14]{LL} this minimum is attained at some $\chi_D$ with
$$
\big\{G(u)<t\big\}\cap\Omega\subset D\subset\big\{G(u)\le t\big\}\cap\Omega
$$
for some $t\ge0$, and under \eqref{Liebermann} the function $G$ is strictly increasing on $[0,\infty)$, so that these sets are $\{u<\tau\}\cap\Omega$ and $\{u\le\tau\}\cap\Omega$ for a suitable $\tau\ge0$. Finally $\int_\Omega G(u)\chi_D\,dx\le\int_\Omega G(u)\varphi\,dx$ gives $\lambda_{s,\mu}(\chi_D,\sigma)\le\Lambda_{s,\mu}(\alpha,\sigma)$, and the reverse inequality holds because $\chi_D\in\mathcal{B}_\alpha$; hence $(u,\chi_D)$ is an optimal pair as well.
\end{proof}

\begin{obs}
As in \cite[Remark 3.6]{FBS}, a further application of the bathtub principle shows that the optimal potential is unique if and only if $|\{u=\tau\}\cap\Omega|=0$. Observe also that $\tau>0$ precisely when $|\{u=0\}\cap\Omega|<\alpha|\Omega|$; in the local case this is automatic, since there the extremals are positive almost everywhere by the strong minimum principle.
\end{obs}

\subsection{Connection between the hard and the soft problems}\label{subsec:connection}
In this subsection we analyze, in the same spirit as in \cite[Theorem 3.7]{FBS}, the connection between the constants $\Lambda_{s,\mu}(\alpha)$ and $\Lambda_{s,\mu}(\alpha,\sigma)$ when $\sigma \to \infty$.

Recall that, by \eqref{Lambdasoft}, $\Lambda_{s,\mu}(\alpha,\sigma)$ is the infimum of $\lambda_{s,\mu}(\varphi,\sigma)$ over the potentials $\varphi\in\mathcal{B}_\alpha$. Given $\varphi\in \mathcal{B}_{\alpha}$ we consider the functional $\mathcal{S}_{s,\varphi,\sigma}\colon W_0^{s,G}(\Omega)\to \R$ defined as
$$
\mathcal{S}_{s,\varphi,\sigma}(v):=\iint_{\R^{2n}}G(D^sv)\,d\nu_n+\sigma\int_\Omega G(v) \varphi\,dx.
$$
\begin{teo}\label{teo:conexion}
Let $0<s<1$, $\mu>0$ and $\alpha\in(0,1)$, let $\Omega\subset \R^n$ be a bounded open domain and let $\Lambda_{s,\mu}(\alpha)$ and $\Lambda_{s,\mu}(\alpha,\sigma)$ be the constants defined as \eqref{Lambdamu} and \eqref{Lambdasoft} respectively. Then
$$
\lim_{\sigma\to \infty}\Lambda_{s,\mu}(\alpha,\sigma)=\Lambda_{s,\mu}(\alpha).
$$
Moreover, let $\{\sigma_j\}_{j\in\N}$ be a sequence of positive numbers with $\sigma_j\uparrow\infty$ and, for every $j\in\N$, let $(u_j,\chi_{D_j})\in W_0^{s,G}(\Omega)\times\mathcal{B}_\alpha$ be an optimal pair for $\Lambda_{s,\mu}(\alpha,\sigma_j)$ with $u_j\geq 0$. Then there exist a subsequence, still denoted by $\{(u_j,\chi_{D_j})\}_{j\in\N}$, a nonnegative extremal $u$ for $\Lambda_{s,\mu}(\alpha)$ and a set $D\in\mathcal{A}_\alpha$ such that
$$
u_j\to u\ \text{ strongly in } W_0^{s,G}(\Omega),\qquad |D_j\symdiff D|\to 0,\qquad D=\{u=0\}\cap\Omega .
$$
In particular $(u,D)$ is an optimal configuration for $\Lambda_{s,\mu}(\alpha)$.
\end{teo}
\begin{proof}
Let $(w,A)$ be an optimal configuration for $\Lambda_{s,\mu}(\alpha)$, given by Theorem \ref{exop}; then $w\in W_A^{s,G}(\Omega)$, $\Gcal(w)=\mu$, $|A|=\alpha |\Omega|$ and
$$
\Lambda_{s,\mu}(\alpha)=\Phi_{s,G}(w).
$$
Therefore, for every $\sigma>0$, since $w=0$ a.e. in $A$,
$$
\Lambda_{s,\mu}(\alpha,\sigma)\leq \lambda_{s,\mu}(\chi_A,\sigma)\leq \mathcal{S}_{s,\chi_{A},\sigma}(w)=\Phi_{s,G}(w),
$$
and so $\Lambda_{s,\mu}(\alpha,\sigma)\leq \Lambda_{s,\mu}(\alpha)$ for every $\sigma>0$. Since $\sigma\mapsto \Lambda_{s,\mu}(\alpha,\sigma)$ is nondecreasing, by Remark \ref{obs:funcionalsoft}(3), there exists $\Lambda^*$ such that
$$
\Lambda_{s,\mu}(\alpha,\sigma)\uparrow \Lambda^*\leq \Lambda_{s,\mu}(\alpha).
$$
Now let $\{\sigma_j\}_{j\in \N}$ be a sequence of positive numbers with $\sigma_j\uparrow \infty$ and, for every $j\in\N$, let $(u_j,\varphi_j)$ be an optimal pair for $\Lambda_{s,\mu}(\alpha,\sigma_j)$ with $u_j\geq 0$ and $\varphi_j=\chi_{D_j}$, which exists by Theorem \ref{teo:paroptimo}. Then
$$
\Phi_{s,G}(u_j)\leq \mathcal{S}_{s,\varphi_j,\sigma_j}(u_j)=\Lambda_{s,\mu}(\alpha,\sigma_j)\leq \Lambda_{s,\mu}(\alpha),
$$
so that $\{u_j\}_{j\in \N}$ is bounded in $W_0^{s,G}(\Omega)$ by Remark \ref{obs:normamodular}. Hence, by the reflexivity of the space $W_0^{s,G}(\Omega)$, the compactness of the embedding $W_0^{s,G}(\Omega) \hookrightarrow L_G(\Omega)$ and the fact that $L^\infty(\Omega)$ is the dual of the separable space $L^1(\Omega)$, there exist a pair $(u,\varphi)\in W_0^{s,G}(\Omega)\times \mathcal{B}_\alpha$ and a subsequence, that we still call $\{(u_j,\varphi_j)\}_{j\in\N}$, such that
\begin{align}
u_j\rightharpoonup u&\quad\text{weakly in } W^{s,G}(\R^n),\label{ujdebil}
\\
u_j\to u &\quad\text{strongly in } L_G(\Omega),\label{ujfuerte}
\\
\varphi_j \stackrel{*}{\rightharpoonup} \varphi &\quad\text{weakly--}*\text{ in } L^\infty(\Omega),\label{ujdebilestrella}
\end{align}
where $\varphi\in\mathcal{B}_\alpha$ follows by testing \eqref{ujdebilestrella} with the functions $\chi_E$, $E\subset\Omega$ measurable, and with $1$, as in the proof of Theorem \ref{teo:paroptimo}.

By \eqref{ujfuerte} and Lemma \ref{lema:convG} we have
$$
\Gcal(u)=\lim_{j\to\infty}\Gcal(u_j)=\mu ,
$$
and the argument that gives \eqref{convcruzado} in the proof of Theorem \ref{teo:paroptimo}, which uses only \eqref{ujfuerte}, \eqref{ujdebilestrella} and Lemma \ref{lema:convG}, yields
\begin{equation}\label{convcruzado2}
\int_\Omega G(u_j)\,\varphi_j\,dx\longrightarrow \int_\Omega G(u)\,\varphi\,dx .
\end{equation}
We claim that
\begin{equation}\label{penalcero}
\int_\Omega G(u)\,\varphi\,dx=0 .
\end{equation}
Indeed, since $\Phi_{s,G}(u_j)\geq 0$,
$$
\sigma_j\int_\Omega G(u_j)\,\varphi_j\,dx\leq \mathcal{S}_{s,\varphi_j,\sigma_j}(u_j)=\Lambda_{s,\mu}(\alpha,\sigma_j)\leq \Lambda_{s,\mu}(\alpha),
$$
so that
$$
\int_\Omega G(u_j)\,\varphi_j\,dx\leq \frac{\Lambda_{s,\mu}(\alpha)}{\sigma_j}\longrightarrow 0
$$
because $\sigma_j\uparrow\infty$ and $\Lambda_{s,\mu}(\alpha)<\infty$, and \eqref{penalcero} follows from \eqref{convcruzado2}.

Since $G(u)\geq 0$, $\varphi\geq 0$ and $G(t)=0$ if and only if $t=0$, \eqref{penalcero} gives $u=0$ almost everywhere in $\{\varphi>0\}$. On the other hand, $\varphi\leq 1$ implies
$$
\alpha|\Omega|=\int_\Omega \varphi\,dx=\int_{\{\varphi>0\}}\varphi\,dx\leq \big|\{\varphi>0\}\big| ,
$$
and therefore
$$
\big|\{u=0\}\cap\Omega\big|\geq \big|\{\varphi>0\}\big|\geq \alpha|\Omega| .
$$
Hence $u$ is admissible in \eqref{caractLambda_s} and, since $\Phi_{s,G}$ is convex and strongly lower semicontinuous, hence weakly lower semicontinuous, \eqref{ujdebil} gives
$$
\Lambda_{s,\mu}(\alpha)\leq \Phi_{s,G}(u)\leq \liminf_{j\to\infty}\Phi_{s,G}(u_j)\leq \liminf_{j\to\infty}\Lambda_{s,\mu}(\alpha,\sigma_j)=\Lambda^* .
$$
Together with $\Lambda^*\leq \Lambda_{s,\mu}(\alpha)$ this gives $\Lambda^*=\Lambda_{s,\mu}(\alpha)$, that is,
$$
\lim_{\sigma\to\infty}\Lambda_{s,\mu}(\alpha,\sigma)=\Lambda_{s,\mu}(\alpha).
$$
It remains to identify the limit pair. The chain above shows that $\Phi_{s,G}(u)=\Lambda_{s,\mu}(\alpha)$, so that $u$ is a nonnegative extremal for $\Lambda_{s,\mu}(\alpha)$ and, by Theorem \ref{teo:identif},
$$
\big|\{u=0\}\cap\Omega\big|=\alpha|\Omega| .
$$
Let $D:=\{\varphi>0\}$. Since $D\subset \{u=0\}\cap\Omega$ and $|D|\geq \alpha|\Omega|$, the two sets coincide up to a null set and $|D|=\alpha|\Omega|$, that is, $D=\{u=0\}\cap\Omega\in\mathcal{A}_\alpha$; moreover $\varphi\leq 1$ and $\int_\Omega\varphi\,dx=\alpha|\Omega|=|D|$ force $\varphi=\chi_D$. Consequently
$$
\lambda_{s,\mu}(D)\leq \Phi_{s,G}(u)=\Lambda_{s,\mu}(\alpha)\leq \lambda_{s,\mu}(D),
$$
so that $(u,D)$ is an optimal configuration for $\Lambda_{s,\mu}(\alpha)$. Testing \eqref{ujdebilestrella} with $\chi_D\in L^1(\Omega)$ and using that $|D_j|=|D|=\alpha|\Omega|$ we obtain
$$
|D_j\symdiff D|=|D_j|+|D|-2|D_j\cap D|=2\alpha|\Omega|-2\int_\Omega \chi_{D_j}\chi_D\,dx\longrightarrow 0 .
$$
Finally we prove that the convergence of $\{u_j\}_{j\in\N}$ is strong. On the one hand $\Phi_{s,G}(u_j)\leq \Lambda_{s,\mu}(\alpha,\sigma_j)\leq \Lambda_{s,\mu}(\alpha)$ and, by weak lower semicontinuity, $\Phi_{s,G}(u)\leq \liminf_{j\to\infty}\Phi_{s,G}(u_j)$, so that
$$
\lim_{j\to\infty}\Phi_{s,G}(u_j)=\Lambda_{s,\mu}(\alpha)=\Phi_{s,G}(u).
$$
On the other hand, by \eqref{ujfuerte} and after passing to a further subsequence, $u_j\to u$ a.e. in $\R^n$, whence $D^su_j\to D^su$ a.e. in $\R^{2n}\setminus\Delta$; moreover $G\in\Delta_2$ by \eqref{LyDelta2}, so that the modulars $\iint_{\R^{2n}}G(D^s(u_j-u))\,d\nu_n$ are bounded. The general form of the Brezis--Lieb lemma, \cite[Theorem 2]{BL}, applied to the function $G$ with respect to the measure $\nu_n$ and with the pointwise inequality $|G(a+b)-G(a)|\leq \varepsilon G(a)+C_\varepsilon G(b)$ of the proof of Lemma \ref{lema:convG}, then gives
$$
\iint_{\R^{2n}}G\big(D^s(u_j-u)\big)\,d\nu_n=\Phi_{s,G}(u_j)-\Phi_{s,G}(u)+o(1)\longrightarrow 0 ,
$$
and therefore $[u_j-u]_{s,G}\to0$ by Remark \ref{obs:normamodular}. Together with \eqref{ujfuerte} this shows that $u_j\to u$ strongly in $W_0^{s,G}(\Omega)$, and the proof is complete.
\end{proof}
\section{Dependence on the energy level $\mu$}\label{sec:asymptotic}
In this section we study the dependence of the constant \eqref{Lambdamu}, and of its optimal configurations, on the energy level $\mu$. Throughout this section, and unless otherwise stated, we assume that $\Omega\subset\R^n$ is open and bounded and that $G$ satisfies \eqref{Liebermann}.
\subsection{Monotonicity and continuity}\label{subsec:monotonia}
We begin by analyzing the behavior of $\Lambda_{s,\mu}(\alpha)$ as $\mu\to\mu_0$ for $0<\mu_0<+\infty$. The following result compares the constants at two different energy levels; since $\Phi_{s,G}$ is not homogeneous the comparison is not exact, and the two exponents $p^-/p^+$ and $p^+/p^-$ appear.

\begin{prop}\label{prop:escala}
Let $s,\alpha\in(0,1)$ be fixed. Then, for every $\mu,\mu_0\in(0,+\infty)$,
\begin{equation}\label{sandwichLambdas}
\min\big\{\theta^{\frac{p^-}{p^+}},\theta^{\frac{p^+}{p^-}}\big\}\,\Lambda_{s,\mu_0}(\alpha)\ \le\ \Lambda_{s,\mu}(\alpha)\ \le\ \max\big\{\theta^{\frac{p^-}{p^+}},\theta^{\frac{p^+}{p^-}}\big\}\,\Lambda_{s,\mu_0}(\alpha), \text{ where }\theta:=\frac{\mu}{\mu_0}.
\end{equation}
In particular $\mu\mapsto\Lambda_{s,\mu}(\alpha)$ is continuous in $(0,+\infty)$.
\end{prop}
\begin{proof}
Assume first $\mu<\mu_0$, so that $\theta<1$, and let $u$ be an extremal for $\Lambda_{s,\mu_0}(\alpha)$, that is $\Gcal(u)=\mu_0$ and $\Phi_{s,G}(u)=\Lambda_{s,\mu_0}(\alpha)$. The function $t\mapsto\Gcal(tu)$ is continuous, vanishes at $t=0$ and equals $\mu_0$ at $t=1$, so that there exists $t_0\in(0,1)$ with $\Gcal(t_0u)=\mu$. Since the rescaling does not change the zero set, $\{t_0u=0\}=\{u=0\}$, the function $t_0u$ is admissible in the characterization \eqref{caractLambda_s} and therefore
$$
\Lambda_{s,\mu}(\alpha)\le\Phi_{s,G}(t_0u).
$$
By \eqref{sandwich} applied to $\Gcal$, from $t_0\le1$ we get $t_0^{p^+}\mu_0\le\Gcal(t_0u)=\theta\mu_0\le t_0^{p^-}\mu_0$, that is
$$
t_0^{p^+}\le\theta\le t_0^{p^-},
$$
and in particular $t_0\le\theta^{1/p^+}$. Using \eqref{sandwich} once more, now for the modular $\Phi_{s,G}$,
$$
\Lambda_{s,\mu}(\alpha)\le\Phi_{s,G}(t_0u)\le t_0^{p^-}\Phi_{s,G}(u)\le\theta^{\frac{p^-}{p^+}}\Lambda_{s,\mu_0}(\alpha).
$$
For the reverse inequality let $u$ be an extremal for $\Lambda_{s,\mu}(\alpha)$ and let $t_1>1$ be such that $\Gcal(t_1u)=\mu_0$, which exists because $\Gcal(tu)\to+\infty$ as $t\to\infty$. As before $t_1u$ is admissible in \eqref{caractLambda_s}, and \eqref{sandwich} gives
$$
t_1^{p^-}\le\frac{1}{\theta}\le t_1^{p^+},
$$
so that $t_1\le\theta^{-1/p^-}$ and
$$
\Lambda_{s,\mu_0}(\alpha)\le\Phi_{s,G}(t_1u)\le t_1^{p^+}\Phi_{s,G}(u)\le\theta^{-\frac{p^+}{p^-}}\Lambda_{s,\mu}(\alpha).
$$
We have thus proved that
$$
\theta^{\frac{p^+}{p^-}}\Lambda_{s,\mu_0}(\alpha)\le\Lambda_{s,\mu}(\alpha)\le\theta^{\frac{p^-}{p^+}}\Lambda_{s,\mu_0}(\alpha)\qquad\text{for }\theta<1.
$$
The case $\mu>\mu_0$ follows by exchanging the roles of $\mu$ and $\mu_0$, and the case $\mu=\mu_0$ is trivial. Finally, both bounds in \eqref{sandwichLambdas} tend to $\Lambda_{s,\mu_0}(\alpha)$ as $\theta\to1$, which gives the continuity.
\end{proof}
\begin{obs}\label{monoLambda}
Proposition \ref{prop:escala} gives monotonicity of the functions
$$
\mu\mapsto \frac{\Lambda_{s,\mu}(\alpha)}{\mu^{\frac{p^-}{p^+}}}\text{ and }\mu\mapsto \frac{\Lambda_{s,\mu}(\alpha)}{\mu^{\frac{p^+}{p^-}}},
$$
more precisely if $\mu_0<\mu$ then
$$
\left (\frac{\mu}{\mu_0}\right )^{\frac{p^-}{p^+}}\Lambda_{s,\mu_0}(\alpha)\leq \Lambda_{s,\mu}(\alpha)\leq \left (\frac{\mu}{\mu_0}\right )^{\frac{p^+}{p^-}}\Lambda_{s,\mu_0}(\alpha),
$$
hence
$$
\frac{\Lambda_{s,\mu_0}(\alpha)}{\mu_0^{\frac{p^-}{p^+}}}\leq \frac{\Lambda_{s,\mu}(\alpha)}{\mu^{\frac{p^-}{p^+}}}\text{ and }\frac{\Lambda_{s,\mu}(\alpha)}{\mu^{\frac{p^+}{p^-}}}\leq \frac{\Lambda_{s,\mu_0}(\alpha)}{\mu_0^{\frac{p^+}{p^-}}}.
$$
That is the functions $\mu\mapsto \frac{\Lambda_{s,\mu}(\alpha)}{\mu^{\frac{p^-}{p^+}}}\text{ and }\mu\mapsto \frac{\Lambda_{s,\mu}(\alpha)}{\mu^{\frac{p^+}{p^-}}}$ are nondecreasing and nonincreasing respectively.
\end{obs}
Remark \ref{monoLambda} gives the limit values of $\Lambda_{s,\mu}(\alpha)$ at the endpoints $0$ and $\infty$. Indeed, putting $\mu=1$ we have
$$
\Lambda_{s,\mu_0}(\alpha)\leq \mu_0^{\frac{p^-}{p^+}}\Lambda_{s,1}(\alpha)\qquad \text{for }0<\mu_0<1,
$$
and so
$$
\lim_{\mu_0\downarrow 0}\Lambda_{s,\mu_0}(\alpha)=0.
$$
Putting $\mu_0=1$ we obtain
$$
\mu^{\frac{p^-}{p^+}}\Lambda_{s,1}(\alpha)\leq \Lambda_{s,\mu}(\alpha)\qquad \text{for }1<\mu<+\infty,
$$
hence
$$
\lim_{\mu\to\infty}\Lambda_{s,\mu}(\alpha)=\infty.
$$
An interesting consequence of Proposition \ref{prop:escala} is contained in the following corollary.
\begin{cor}\label{propLambda}
Let $s,\alpha\in (0,1)$. Then
\begin{enumerate}
\item $\Lambda_{s,\mu}(\alpha)$ is strictly increasing in $\mu$.
\item $\Lambda_{s,\mu}(\alpha)$ is locally Lipschitz continuous in $\mu$.
\end{enumerate}
\end{cor}
\begin{proof}
Suppose that $\mu_0<\mu$. Then Proposition \ref{prop:escala} with $\theta=\mu_0/\mu<1$ and the mean value theorem imply that there exists $\eta\in (\theta,1)$ such that
\begin{equation}\label{paracoro}
0< (1-\theta^{\frac{p^-}{p^+}})\Lambda_{s,\mu}(\alpha)\leq \Lambda_{s,\mu}(\alpha)-\Lambda_{s,\mu_0}(\alpha)\leq (1-\theta^{\frac{p^+}{p^-}})\Lambda_{s,\mu}(\alpha)=\frac{p^+}{p^-}\eta^{\frac{p^+}{p^-}-1}\frac{(\mu-\mu_0)}{\mu}\Lambda_{s,\mu}(\alpha).
\end{equation}
The left hand side in \eqref{paracoro} establishes (1). Suppose now that $[m,M]\subset (0,+\infty)$ is an arbitrary interval and that $\mu_0,\mu\in [m,M]$ with $\mu_0<\mu$. Since $\eta<1$ and $p^-\leq p^+$ we have $\eta^{\frac{p^+}{p^-}-1}\leq 1$, so that \eqref{paracoro}, the monotonicity (1) and the bounds $\mu\geq m$ and $\Lambda_{s,\mu}(\alpha)\leq \Lambda_{s,M}(\alpha)$ imply
\begin{equation}\label{lipycota}
|\Lambda_{s,\mu}(\alpha)-\Lambda_{s,\mu_0}(\alpha)|\leq \frac{p^+\Lambda_{s,M}(\alpha)}{p^-m}|\mu-\mu_0|,
\end{equation}
and the case $\mu<\mu_0$ follows by exchanging the roles of $\mu$ and $\mu_0$. Therefore we can conclude (2) with Lipschitz constant $L:=\frac{p^+\Lambda_{s,M}(\alpha)}{p^-m}$.
\end{proof}
\subsection{The asymptotic behavior of $\Lambda_{s,\mu}(\alpha)/\mu$}\label{subsec:asintmu}
In order to analyze the precise asymptotic behavior of the quotients $\Lambda_{s,\mu}(\alpha)/\mu$ at the endpoints $0$ and $\infty$ we introduce for $s,\alpha\in(0,1)$, a Young function $M$ and $\tau>0$ the constants
\begin{align*}
E_W^{M}(\tau,\alpha)&:=\inf \left \{\iint_{\R^{2n}}M(D^su)\,d\nu_n\colon  u\in W_0^{s,M}(\Omega),|\{u=0\}\cap\Omega|\geq \alpha |\Omega|, \int_\Omega M(u)\,dx=\tau\right \}
\\
E_C^{M}(\tau,\alpha)&:=\inf \left \{\iint_{\R^{2n}}M(D^su)\,d\nu_n\colon  u\in C_c^{\infty}(\Omega),|\{u=0\}\cap\Omega|\geq \alpha |\Omega|, \int_\Omega M(u)\,dx=\tau\right \}.
\end{align*}
Both classes are nonempty --- it is enough to take a smooth function supported in a small ball contained in $\Omega$, suitably rescaled --- so that $E_W^M(\tau,\alpha)\leq E_C^M(\tau,\alpha)<+\infty$, and both constants are positive by Proposition \ref{prop:poincare} applied to $M$. The following result establishes the appropriate bounds for the $\liminf$ and the $\limsup$ of the quotients.
\begin{teo}\label{cotascociente}
Let $s,\alpha\in (0,1)$ be fixed. Assume that the limits defining $M_i(t)$ for $i=0,\infty$ exist, and that for every family $\{u_\mu\}$ of extremals for $\Lambda_{s,\mu}(\alpha)$ the normalized densities
\begin{equation}\label{NC}\tag{NC}
\left\{\frac{G(u_\mu)}{\mu}\right\}\qquad\text{are uniformly integrable in }\Omega
\end{equation}
(see Remark \ref{obs:NC} for sufficient conditions). Then
\begin{equation}\label{chain1}\tag{CH$0$}
\inf_{\tau>0}\frac{E_W^{M_0}(\tau,\alpha)}{\tau}\leq\liminf_{\mu\downarrow 0}\frac{\Lambda_{s,\mu}(\alpha)}{\mu}\leq\limsup_{\mu\downarrow 0}\frac{\Lambda_{s,\mu}(\alpha)}{\mu}\le\inf_{\tau>0}\frac{E_C^{M_0}(\tau,\alpha)}{\tau}
\end{equation}
and
\begin{equation}\label{chain2}\tag{CH$\infty$}
\inf_{\tau>0}\frac{E_W^{M_\infty}(\tau,\alpha)}{\tau}\leq\liminf_{\mu\to \infty}\frac{\Lambda_{s,\mu}(\alpha)}{\mu}\leq\limsup_{\mu\to \infty}\frac{\Lambda_{s,\mu}(\alpha)}{\mu}\le\inf_{\tau>0}\frac{E_C^{M_\infty}(\tau,\alpha)}{\tau}.
\end{equation}
\end{teo}
\begin{proof}
We first prove \eqref{chain1}.

\emph{Upper bound.} Let $u\in C_c^\infty(\Omega)$, $u\not\equiv0$, with $|\{u=0\}\cap\Omega|\ge\alpha|\Omega|$. The map $k\mapsto\int_\Omega G(ku)\,dx$ is continuous, vanishes at $k=0$ and tends to $\infty$ as $k\to\infty$, so there exists $k(\mu)>0$ with $\int_\Omega G(k(\mu)u)\,dx=\mu$; moreover $k(\mu)\to0$ as $\mu\downarrow0$, since $\mu\ge\min\{k(\mu)^{p^-},k(\mu)^{p^+}\}\Gcal(u)$ by \eqref{sandwich}. Since $C_c^\infty(\Omega)\subset W_0^{s,G}(\Omega)$ and the rescaling does not change the zero set, $k(\mu)u$ is admissible in \eqref{caractLambda_s} and therefore
$$
\frac{\Lambda_{s,\mu}(\alpha)}{\mu}\le\frac1\mu\iint_{\R^{2n}}G(k(\mu)D^su)\,d\nu_n
=\frac{\displaystyle\iint_{\R^{2n}}\frac{G(k(\mu)D^su)}{G(k(\mu))}\,d\nu_n}{\displaystyle\int_\Omega\frac{G(k(\mu)u)}{G(k(\mu))}\,dx}.
$$
By \eqref{sandwich} the two integrands are bounded, uniformly in $k(\mu)$, by $|D^su|^{p^-}+|D^su|^{p^+}$ and by $|u|^{p^-}+|u|^{p^+}$, which are integrable with respect to $\nu_n$ and to $dx$ respectively because $u\in C_c^\infty(\Omega)$ and $s<1$. Since the limit defining $M_0$ exists, dominated convergence gives
$$
\limsup_{\mu\downarrow0}\frac{\Lambda_{s,\mu}(\alpha)}{\mu}\le\frac{\iint_{\R^{2n}}M_0(D^su)\,d\nu_n}{\int_\Omega M_0(u)\,dx},
$$
and taking the infimum over all such $u$ with $\int_\Omega M_0(u)\,dx=\tau$, and then over $\tau>0$,
\begin{equation}\label{c1}
\limsup_{\mu\downarrow0}\frac{\Lambda_{s,\mu}(\alpha)}{\mu}\le\inf_{\tau>0}\frac{E_C^{M_0}(\tau,\alpha)}{\tau}.
\end{equation}
\emph{Lower bound.} By \eqref{c1} there are $\mu_0\in(0,1)$ and $C>0$ with $\Lambda_{s,\mu}(\alpha)/\mu\le C$ for every $0<\mu\le\mu_0$. For such $\mu$ let $u_\mu$ be an extremal for $\Lambda_{s,\mu}(\alpha)$, given by Theorem \ref{exop}, let $r(\mu)>0$ be defined by $G(r(\mu))=\mu$ (observe that $r(\mu)\downarrow 0$ if $\mu\downarrow 0$) and put $v_\mu:=u_\mu/r(\mu)$. Since $r(\mu)v_\mu=u_\mu$ and $G(r(\mu))=\mu$,
\begin{equation}\label{normaliza}
\iint_{\R^{2n}}\frac{G(r(\mu)D^sv_\mu)}{G(r(\mu))}\,d\nu_n=\frac{\Lambda_{s,\mu}(\alpha)}{\mu}\le C,
\qquad
\int_\Omega\frac{G(r(\mu)v_\mu)}{G(r(\mu))}\,dx=\frac{\Gcal(u_\mu)}{\mu}=1 .
\end{equation}
Let $\Psi^*$ be the convex envelope of $\min\{t^{p^-},t^{p^+}\}$; it is a Young function satisfying \eqref{Liebermann} with the same exponents, and \eqref{sandwich} evaluated at $r(\mu)$ gives $\Psi^*(t)\le G(r(\mu)t)/G(r(\mu))$ for every $t>0$. Hence the modulars $\iint_{\R^{2n}}\Psi^*(|D^sv_\mu|)\,d\nu_n$ are bounded by $C$, and Proposition \ref{prop:poincare} and Remark \ref{obs:normamodular} applied to $\Psi^*$ show that $\{v_\mu\}_{0<\mu\le\mu_0}$ is bounded in $W_0^{s,\Psi^*}(\Omega)$.

Choose $\mu_j\downarrow0$ with $\Lambda_{s,\mu_j}(\alpha)/\mu_j\to\liminf_{\mu\downarrow0}\Lambda_{s,\mu}(\alpha)/\mu$ and write $v_j:=v_{\mu_j}$, $r_j:=r(\mu_j)$. By Proposition \ref{prop:espacios} applied to $\Psi^*$ the embedding $W_0^{s,\Psi^*}(\Omega)\hookrightarrow L_{\Psi^*}(\Omega)$ is compact, so that we may assume, after passing to a subsequence, that $v_j\to v$ in $L_{\Psi^*}(\Omega)$ and almost everywhere in $\Omega$, in particular $D^sv_j\to D^sv$ a.e. in $\R^{2n}$. The functions $t\mapsto G(rt)/G(r)$ are convex and converge pointwise to $M_0$, hence uniformly on compact subsets of $[0,\infty)$, so that
$$
\frac{G(r_jD^sv_j)}{G(r_j)}\to M_0(D^sv)\ \text{ a.e. in }\R^{2n},
\qquad
\frac{G(r_jv_j)}{G(r_j)}=\frac{G(u_{\mu_j})}{\mu_j}\to M_0(v)\ \text{ a.e. in }\Omega .
$$
Fatou's lemma applied to the first convergence gives
\begin{equation}\label{fatounum}
\liminf_{\mu\downarrow0}\frac{\Lambda_{s,\mu}(\alpha)}{\mu}=\lim_{j\to\infty}\iint_{\R^{2n}}\frac{G(r_jD^sv_j)}{G(r_j)}\,d\nu_n\ \ge\ \iint_{\R^{2n}}M_0(D^sv)\,d\nu_n ,
\end{equation}
so that $v\in W_0^{s,M_0}(\Omega)$; by \eqref{NC}, Vitali's convergence theorem and \eqref{normaliza}, the second one gives $\int_\Omega M_0(v)\,dx=1$. Finally, writing $E_j:=\{v_j=0\}\cap\Omega$, the reverse Fatou lemma yields $|\limsup_jE_j|\ge\limsup_j|E_j|\ge\alpha|\Omega|$, and $v=0$ almost everywhere in $\limsup_jE_j$, whence $|\{v=0\}\cap\Omega|\ge\alpha|\Omega|$. Therefore $v$ is admissible in the definition of $E_W^{M_0}(1,\alpha)$ and, by \eqref{fatounum},
$$
\inf_{\tau>0}\frac{E_W^{M_0}(\tau,\alpha)}{\tau}\le E_W^{M_0}(1,\alpha)\le\iint_{\R^{2n}}M_0(D^sv)\,d\nu_n\le\liminf_{\mu\downarrow0}\frac{\Lambda_{s,\mu}(\alpha)}{\mu}.
$$
The proof of \eqref{chain2} is identical: since \eqref{sandwich} holds at every scale, both the dominating function of the upper bound and the minorant $\Psi^*$ of the lower bound remain available when $k(\mu),r(\mu)\to\infty$; one only replaces $M_0$ by $M_\infty$ and reads the threshold $\mu_0$ as a lower bound.
\end{proof}
\begin{obs}\label{obs:NC}
Observe that condition \eqref{NC} is only used in the two lower bounds; the upper bounds hold without it. It is a non-concentration requirement on the extremals, and it is not needed when $G(t)=t^p$: in that case $G(u_\mu)/\mu=|v_\mu|^p$, the subsequence provided by the compact embedding converges strongly in $L^p(\Omega)$ and therefore $\int_\Omega|v|^p\,dx=\lim_{j\to\infty}\int_\Omega|v_j|^p\,dx=1$, with no need of Vitali's theorem. More generally, \eqref{NC} holds whenever the rescaled extremals are bounded in $L^q(\Omega)$ for some $q>p^+$, because then the bound $G(u_\mu)/\mu\le |v_\mu|^{p^-}+|v_\mu|^{p^+}$ coming from \eqref{sandwich} provides an equi--integrable majorant.

Let us observe that in the unconstrained problem no hypothesis of this kind is needed. There the extremals solve the Euler--Lagrange equation in the whole of $\Omega$, so that the global regularity of eigenfunctions of the fractional $g$--Laplacian \cite{FBSalortVivas} applies and yields $L^\infty$ bounds for them, from which \eqref{NC} follows as soon as such bounds are uniform along the family. In the presence of the obstacle the test functions must vanish on $A$, no equation is available in the whole of $\Omega$, and this is precisely why \eqref{NC} is imposed here. We expect it to be removable by a concentration--compactness argument in the spirit of \cite{FBSilva}; we do not pursue this here.
\end{obs}
Theorem \ref{cotascociente} gives bounds for the lower and upper limits, but what we really expect is to identify the limit of the quotients. To this end we need the following two important lemmas.
\begin{lema}\label{lema:densidad}
Let $s\in (0,1)$ and $\tau>0$ be fixed, let $\Omega\subset\R^n$ be a bounded Lipschitz domain and let $M$ be a Young function satisfying \eqref{Liebermann}. Then
$$
E_W^{M}(\tau,\alpha)=E_C^{M}(\tau,\alpha),
$$
for every $\alpha\in(0,1)$ at which the function $\beta \mapsto E_C^{M}(\tau,\beta)$ is left-continuous. Since this function is nondecreasing, the identity holds for all but countably many $\alpha\in(0,1)$.
\end{lema}
\begin{proof}
Observe that by definition $E_W^{M}(\tau,\alpha)\leq	E_C^{M}(\tau,\alpha)$. For the reverse inequality let $u$ be an extremal for $E_W^M(\tau,\alpha)$, which exists by Theorem \ref{exop} applied to $M$ and can be taken nonnegative by Lemma \ref{lema:signo}, and set $Z:=\{u=0\}\cap\Omega$, so that $|Z|\ge\alpha|\Omega|$.

Since $\partial\Omega$ is Lipschitz and $M$ satisfies \eqref{Liebermann}, the space $W_0^{s,M}(\Omega)$ is the closure of $C_c^\infty(\Omega)$ in $W^{s,M}(\R^n)$; see the comments after Proposition \ref{prop:espacios}. Hence there are functions $\varphi_j\in C_c^\infty(\Omega)$ such that $\varphi_j\to u$ in $W^{s,M}(\R^n)$ and, after passing to a subsequence, a.e. in $\R^n$. Applying Lemma \ref{lema:convG} with $dx$ and with $\nu_n$, and recalling Remark \ref{obs:normamodular}, we obtain
\begin{equation}\label{convmodulares}
\int_\Omega M(\varphi_j)\,dx\to \int_\Omega M(u)\,dx=\tau,\qquad \iint_{\R^{2n}}M(D^s\varphi_j)\,d\nu_n\to \iint_{\R^{2n}}M(D^su)\,d\nu_n=E_W^M(\tau,\alpha).
\end{equation}
Fix $\delta>0$ and set $w_j:=(\varphi_j-\delta)_+$ and $V_j:=\{\varphi_j<\delta\}$. Since $t\mapsto (t-\delta)_+$ is nondecreasing, $1$--Lipschitz and vanishes on $(-\infty,\delta]$, the function $w_j$ is nonnegative, Lipschitz, and satisfies
$$
\supp w_j\subset \supp \varphi_j,\qquad 0\le w_j\le |\varphi_j|,\qquad |D^sw_j|\le |D^s\varphi_j|,\qquad w_j\equiv0\ \text{ in } V_j,
$$
and the set $V_j$ is open, because $\varphi_j$ is continuous. This is the point of the truncation: it turns the measurable constraint $|\{u=0\}\cap\Omega|\ge\alpha|\Omega|$ into an open one, and open constraints are stable under mollification. Indeed, since $\varphi_j\to u=0$ a.e. in $Z$, for a.e. $x\in Z$ we have $x\in V_j$ for every $j$ large enough, and hence Fatou's lemma gives
\begin{equation}\label{fatouniveles}
\liminf_{j\to\infty}|V_j\cap\Omega|\geq |Z|\geq \alpha|\Omega|.
\end{equation}
Let $\rho_\varepsilon$ be a standard mollifier and, for $0<\varepsilon<d(\supp \varphi_j,\partial\Omega)$, set $w_j^\varepsilon:=w_j\ast\rho_\varepsilon\in C_c^\infty(\Omega)$. Since $D^sw_j^\varepsilon(x,y)=\int_{\R^n}D^sw_j(x-z,y-z)\rho_\varepsilon(z)\,dz$ and $\nu_n$ is invariant under the simultaneous translations $(x,y)\mapsto(x-z,y-z)$, Jensen's inequality yields
\begin{equation}\label{jensenmolli}
\iint_{\R^{2n}}M(D^sw_j^\varepsilon)\,d\nu_n\leq \iint_{\R^{2n}}M(D^sw_j)\,d\nu_n\leq \iint_{\R^{2n}}M(D^s\varphi_j)\,d\nu_n.
\end{equation}
Moreover $w_j^\varepsilon\to w_j$ uniformly as $\varepsilon\downarrow 0$, with supports contained in a fixed compact subset of $\Omega$, so that $\int_\Omega M(w_j^\varepsilon)\,dx\to\int_\Omega M(w_j)\,dx$; and $w_j^\varepsilon$ vanishes on the open set $V_j^\varepsilon:=\{x\in V_j\colon d(x,\partial V_j)>\varepsilon\}$, whose measure increases to $|V_j\cap\Omega|$ as $\varepsilon\downarrow 0$. We may therefore choose, for every $j$, a parameter $\varepsilon_j>0$ such that the function $\psi_j:=w_j^{\varepsilon_j}\in C_c^\infty(\Omega)$ satisfies
$$
|V_j^{\varepsilon_j}\cap\Omega|\geq |V_j\cap\Omega|-\frac1j,\qquad \left |\int_\Omega M(\psi_j)\,dx-\int_\Omega M(w_j)\,dx\right |\leq \frac1j.
$$
Since $|w_j-(u-\delta)_+|\le|\varphi_j-u|$, Lemma \ref{lema:convG} gives $\int_\Omega M(w_j)\,dx\to\int_\Omega M((u-\delta)_+)\,dx=:\tau_\delta$ as $j\to\infty$, and $\tau_\delta\uparrow\tau$ as $\delta\downarrow0$ by monotone convergence; in particular $0<\tau_\delta\le\tau$ for every $\delta$ small enough, which we assume from now on.

Finally we pick $c_j>0$ such that $\int_\Omega M(c_j\psi_j)\,dx=\tau$, which is possible for $j$ large enough because $\int_\Omega M(\psi_j)\,dx\to\tau_\delta>0$. By \eqref{sandwich},
$$
\min\{c_j^{p^-},c_j^{p^+}\}\int_\Omega M(\psi_j)\,dx\leq \tau\leq \max\{c_j^{p^-},c_j^{p^+}\}\int_\Omega M(\psi_j)\,dx,
$$
and the first inequality gives $\kappa_j:=\max\{c_j^{p^-},c_j^{p^+}\}\leq \max\{1,\theta_j^{p^+/p^-}\}$, where $\theta_j:=\tau\big(\int_\Omega M(\psi_j)\,dx\big)^{-1}\to\tau/\tau_\delta$; hence
\begin{equation}\label{tautaudelta}
\limsup_{j\to\infty}\kappa_j\leq \left (\frac{\tau}{\tau_\delta}\right )^{p^+/p^-}.
\end{equation}
The function $c_j\psi_j\in C_c^\infty(\Omega)$ satisfies $\int_\Omega M(c_j\psi_j)\,dx=\tau$ and $|\{c_j\psi_j=0\}\cap\Omega|\geq |V_j^{\varepsilon_j}\cap\Omega|\geq(\alpha-\eta_j)|\Omega|$, where
$$
\eta_j:=\left (\alpha-\frac{|V_j^{\varepsilon_j}\cap\Omega|}{|\Omega|}\right )_+\longrightarrow 0
$$
by \eqref{fatouniveles} and the choice of $\varepsilon_j$. Since $\eta_j\to 0$ as $j\to\infty$, for every $j$ large enough we have $\eta_j<\alpha$, and $c_j\psi_j$ is therefore admissible in the definition of $E_C^M(\tau,\alpha-\eta_j)$, so that by \eqref{sandwich} and \eqref{jensenmolli}
$$
E_C^M(\tau,\alpha-\eta_j)\leq \iint_{\R^{2n}}M(c_jD^s\psi_j)\,d\nu_n\leq \kappa_j\iint_{\R^{2n}}M(D^s\psi_j)\,d\nu_n\leq \kappa_j\iint_{\R^{2n}}M(D^s\varphi_j)\,d\nu_n.
$$
Letting $j\to\infty$, the left-continuity of $\beta\mapsto E_C^M(\tau,\beta)$ at $\alpha$, together with \eqref{tautaudelta} and \eqref{convmodulares}, gives
$$
E_C^M(\tau,\alpha)\leq \left (\frac{\tau}{\tau_\delta}\right )^{p^+/p^-}E_W^M(\tau,\alpha).
$$
Letting now $\delta\downarrow 0$, so that $\tau_\delta\uparrow\tau$, we conclude that $E_C^M(\tau,\alpha)\leq E_W^M(\tau,\alpha)$.
\end{proof}
\begin{obs}
The regularity of $\partial\Omega$ enters only through the density of $C_c^\infty(\Omega)$ in $W_0^{s,M}(\Omega)$, so that Lemma \ref{lema:densidad} holds, with the same proof, for every bounded open set $\Omega$ for which such a density result is available; for example, the bounded open sets with the star-shaped property with respect to a ball contained in $\Omega$.
\end{obs}
Taking into account the above remark we recall the following definition.
\begin{defi}\label{def:estrellado}
Let $\Omega\subset \R^n$ be a domain. We say that $\Omega$ has the property $\mathcal{S}$ if it is bounded and there exists a ball $B=B_r(x_0)\subset \Omega$ such that $\Omega$ is star-shaped with respect to every point of $B$, that is, for every $x\in \Omega$ and every $y\in B$,
$$
tx+(1-t)y\in \Omega\qquad \forall t\in (0,1).
$$
\end{defi}
Two comments on Definition \ref{def:estrellado} are in order. On the one hand, every bounded convex domain has the property $\mathcal{S}$, since it is star-shaped with respect to each of the balls it contains; the property is however strictly weaker than convexity, as an $L$--shaped domain shows. On the other hand, every domain with the property $\mathcal{S}$ is a bounded Lipschitz domain. Indeed, if $x\in\partial\Omega$ then the convex hull of $\{x\}$ and $B$ is contained in $\overline\Omega$, and it contains the cone with vertex $x$, axis $(x_0-x)/|x_0-x|$ and half-aperture $\theta$ given by $\sin\theta=r/|x-x_0|\geq r/\operatorname{diam}(\Omega)$; since this aperture is independent of $x$, the set $\Omega$ enjoys the uniform $\varepsilon$--cone property, which characterizes the Lipschitz regularity of $\partial\Omega$, see \cite[Theorem 2.4.7]{HP}. In particular both Lemma \ref{lema:densidad} and Lemma \ref{continuidadE} are available under the property $\mathcal{S}$, and each of them uses a different part of it: the whole ball is what yields the Lipschitz regularity of $\partial\Omega$, and hence the density of $C_c^\infty(\Omega)$ in $W_0^{s,M}(\Omega)$, while only its center is needed for the rescaling argument, since $\lambda(\Omega-x_0)+x_0\subset\Omega$ for every $\lambda\in(0,1)$. Let us finally stress that the property $\mathcal{S}$ is strictly stronger than the Lipschitz regularity of $\partial\Omega$: an annulus is smooth but it is not star-shaped with respect to any of its points.
\begin{lema}\label{continuidadE}
Let $s\in (0,1)$, $\tau>0$ and let $M$ be a Young function satisfying \eqref{Liebermann}. If $\Omega\subset \R^n$ has the property $\mathcal{S}$ then the function $\alpha \mapsto E_C^{M}(\tau,\alpha)$ is continuous in $(0,1)$.
\end{lema}
\begin{proof}
Assume first that the center of the ball in the property $\mathcal{S}$ is $x_0=0$. Let $0<\beta<\alpha<1$ and let $\varepsilon>0$. There exists $u\in C_c^\infty(\Omega)$ such that
$$
\iint_{\R^{2n}}M(D^s u)\,d\nu_n< E_C^M(\tau,\beta)+\varepsilon,
$$
$|\{u=0\}\cap \Omega|\geq \beta|\Omega|$ and $\int_\Omega M(u)\,dx=\tau$. Putting
$$
\lambda:=\left (\frac{1-\alpha}{1-\beta} \right)^{\frac{1}{n}}
$$
we consider the rescaled function $u_\lambda(x):=u(x/\lambda)$. By the star-shaped property $\mathcal{S}$ we have $\supp(u_\lambda)=\lambda \supp(u)\subset \lambda \Omega\subset \Omega$, and therefore $u_\lambda \in C_c^\infty(\Omega)$ and
\begin{align*}
|\{u_\lambda \neq0\}\cap \Omega|&=|\lambda \{u\neq 0\}\cap\Omega|
\\
&=\lambda^n|\{u\neq 0\}\cap\Omega|
\\
&=\lambda^n(|\Omega|-|\{u=0\}\cap\Omega|)
\\
&\leq \lambda^n(|\Omega|-\beta |\Omega|)
\\
&=(1-\alpha)|\Omega|,
\end{align*}
hence $|\{u_\lambda =0\}\cap\Omega|=|\Omega|-|\{u_\lambda \neq0\}\cap\Omega|\geq \alpha |\Omega|$. Also, we can choose $c>0$ such that $\int_\Omega M(cu_\lambda)\,dx=\tau$ and taking into account that $\int_\Omega M(u_\lambda)\,dx=\lambda^n \int_\Omega M(u)\,dx$ we have from \eqref{sandwich}
$$
\min\{c^{p^-},c^{p^+}\} \leq \left (\frac{1}{\lambda}\right )^{n}\leq \max\{c^{p^-},c^{p^+}\} ,
$$
in particular $1<c\leq 1/\lambda^{n/p^{-}}$. Since $|\{cu_\lambda =0\}\cap\Omega|=|\{u_\lambda =0\}\cap\Omega|$ the above says that $c u_\lambda$ is admissible for $E_C^M(\tau,\alpha)$ and then using change of variable and \eqref{sandwich}
\begin{align*}
E_C^M(\tau,\alpha)&\leq \iint_{\R^{2n}}M(cD^s u_\lambda)\,d\nu_n
\\
&=\lambda^{n}\iint_{\R^{2n}}M(c\lambda ^{-s}D^s u)\,d\nu_n
\\
&\leq \lambda^{n}c^{p^+}\lambda^{-s p^{+}}\iint_{\R^{2n}}M(D^s u)\,d\nu_n
\\
&\leq \lambda^{n-\frac{n p^+}{p^-}-sp^+}\iint_{\R^{2n}}M(D^s u)\,d\nu_n
\\
&= \left (\frac{1-\beta}{1-\alpha}\right )^{\frac{p^+-p^-}{p^-}+\frac{sp^+}{n}}\iint_{\R^{2n}}M(D^s u)\,d\nu_n
\\
&<\left (\frac{1-\beta}{1-\alpha}\right )^{\frac{p^+-p^-}{p^-}+\frac{sp^+}{n}}E_C^M(\tau,\beta)+\left (\frac{1-\beta}{1-\alpha}\right )^{\frac{p^+-p^-}{p^-}+\frac{sp^+}{n}}\varepsilon,
\end{align*}
for all $\varepsilon>0$ and so
$$
E_C^M(\tau,\alpha)\leq \left (\frac{1-\beta}{1-\alpha}\right )^{\frac{p^+-p^-}{p^-}+\frac{sp^+}{n}}E_C^M(\tau,\beta)
$$
with $\gamma:=(p^+-p^-)/p^-+sp^+/n>0$  we obtain that the function $\alpha\mapsto (1-\alpha)^\gamma E_C^M(\tau,\alpha)$ is nonincreasing. On the other hand, the class of admissible functions in the definition of $E_C^M(\tau,\alpha)$ shrinks as $\alpha$ increases, so that $\alpha\mapsto E_C^M(\tau,\alpha)$ is nondecreasing. Both functions being monotone, for every $\alpha\in(0,1)$ the one-sided limits $E_C^M(\tau,\alpha^\pm):=\lim_{\beta\to\alpha^\pm}E_C^M(\tau,\beta)$ exist, and the second monotonicity gives
$$
E_C^M(\tau,\alpha^-)\leq E_C^M(\tau,\alpha)\leq E_C^M(\tau,\alpha^+),
$$
while the first one, since $\beta\mapsto(1-\beta)^\gamma$ is continuous and positive, gives the three reversed inequalities. Hence the three values coincide and $\alpha\mapsto E_C^M(\tau,\alpha)$ is continuous in $(0,1)$. Finally, if $x_0\neq0$ it is enough to observe that both $dx$ and $\nu_n$ are invariant under translations, so that neither the constants $E_C^M(\tau,\cdot)$ nor the property $\mathcal{S}$ are affected when $\Omega$ is replaced by $\Omega-x_0$; one may therefore assume $x_0=0$ without loss of generality.
\end{proof}
With the hypotheses of Lemma \ref{continuidadE} we can introduce the constant $E^M(\tau,\alpha):=E_C^M(\tau,\alpha)=E_W^M(\tau,\alpha)$. Now we are in a position to give the main result of this section, which is also the second main result of the paper.
\begin{teo}\label{limitecociente}
Let $s,\alpha\in (0,1)$ be fixed and $\Omega\subset \R^n$ satisfying property $\mathcal{S}$. Assume that the limits defining $M_i(t)$ for $i=0,\infty$ exist, and that for every family $\{u_\mu\}$ of extremals for $\Lambda_{s,\mu}(\alpha)$ the condition \eqref{NC} is fulfilled. Then
\begin{equation}\label{convmucero}
\lim_{\mu\downarrow 0}\frac{\Lambda_{s,\mu}(\alpha)}{\mu}=\inf_{\tau>0}\frac{E^{M_0}(\tau,\alpha)}{\tau}
\end{equation}
and
\begin{equation}\label{convmuinfinito}
\lim_{\mu\to \infty}\frac{\Lambda_{s,\mu}(\alpha)}{\mu}=\inf_{\tau>0}\frac{E^{M_\infty}(\tau,\alpha)}{\tau}
\end{equation}
hold. Moreover, let $r(\mu)>0$ be defined by $G(r(\mu))=\mu$, for every $\mu>0$ let $u_\mu$ be an extremal for $\Lambda_{s,\mu}(\alpha)$ and let $i\in\{0,\infty\}$. If $\{\mu_j\}_{j\in\N}$ is any sequence of positive numbers with $\mu_j\downarrow 0$ when $i=0$ and $\mu_j\to+\infty$ when $i=\infty$, then the sequence $\{u_{\mu_j}/r(\mu_j)\}_{j\in\N}$ is precompact in $L^{p^-}(\Omega)$ and each of its accumulation points is an extremal for $E^{M_i}(1,\alpha)$.
\end{teo}
\begin{proof}
By Remark \ref{obs:Matuyoung}, $M_0$ and $M_\infty$ are Young functions satisfying \eqref{Liebermann} with the same exponents $p^\pm$, and since $\Omega$ has the property $\mathcal{S}$ it is a bounded Lipschitz domain. Hence Lemma \ref{continuidadE} applied to $M_i$ shows that $\alpha\mapsto E_C^{M_i}(\tau,\alpha)$ is continuous in $(0,1)$ for every $\tau>0$, in particular left-continuous, and Lemma \ref{lema:densidad} then gives
$$
E_W^{M_i}(\tau,\alpha)=E_C^{M_i}(\tau,\alpha)=E^{M_i}(\tau,\alpha)\qquad\text{for every }\tau>0,\ \alpha\in(0,1),\ i=0,\infty .
$$
Taking the infimum over $\tau>0$, the two extreme terms of each of the chains \eqref{chain1} and \eqref{chain2} in Theorem \ref{cotascociente} coincide, and therefore \eqref{convmucero} and \eqref{convmuinfinito} hold.

Finally, let $i$ and $\{\mu_j\}_{j\in\N}$ be as in the statement and write $v_j:=u_{\mu_j}/r(\mu_j)$. By the proof of Theorem \ref{cotascociente} the sequence $\{v_j\}_{j\in\N}$ is bounded in $W_0^{s,\Psi^*}(\Omega)$ and hence, by Proposition \ref{prop:espacios} applied to $\Psi^*$, precompact in $L_{\Psi^*}(\Omega)$; it is then precompact in $L^{p^-}(\Omega)$, since
$$
\int_\Omega|w|^{p^-}\,dx\leq \delta^{p^-}|\Omega|+\frac{\delta^{-(p^+-p^-)}}{\Psi^*(1)}\int_\Omega\Psi^*(w)\,dx\qquad\text{for every }\delta\in(0,1),
$$
as follows by splitting $\Omega$ into $\{|w|\leq\delta\}$ and $\{|w|>\delta\}$ and applying \eqref{sandwich} for $\Psi^*$ twice on the second set. Let $v$ be an accumulation point, so that along a subsequence $v_j\to v$ in $L^{p^-}(\Omega)$ and almost everywhere in $\Omega$. Arguing as in the proof of Theorem \ref{cotascociente}, $v$ is admissible in the definition of $E^{M_i}(1,\alpha)$ and
$$
\iint_{\R^{2n}}M_i(D^sv)\,d\nu_n\leq\liminf_{j\to\infty}\frac{\Lambda_{s,\mu_j}(\alpha)}{\mu_j}=\lim_{\mu\to i}\frac{\Lambda_{s,\mu}(\alpha)}{\mu}=\inf_{\tau>0}\frac{E^{M_i}(\tau,\alpha)}{\tau}=E^{M_i}(1,\alpha),
$$
where the second equality holds because the limit exists by the first part of the theorem, and the last one follows from the homogeneity of $M_i$. Since the reverse inequality holds by admissibility, $v$ is an extremal for $E^{M_i}(1,\alpha)$.
\end{proof}
In Theorem \ref{limitecociente} the limits defining $M_i$ are assumed to exist, so that by Remark \ref{obs:Matuyoung} the functions $M_i$ for $i=0,\infty$ are the powers $M_0(t)=t^{p_0}$ and $M_\infty(t)=t^{p_\infty}$ respectively; in this case we have homogeneity thus $E^{M_0}(\tau,\alpha)=\tau E^{M_0}(1,\alpha)$	and $E^{M_\infty}(\tau,\alpha)=\tau E^{M_\infty}(1,\alpha)$	 therefore
$$
\lim_{\mu\downarrow 0}\frac{\Lambda_{s,\mu}(\alpha)}{\mu}=E^{M_0}(1,\alpha)
$$
and
$$
\lim_{\mu\to \infty}\frac{\Lambda_{s,\mu}(\alpha)}{\mu}=E^{M_\infty}(1,\alpha),
$$
where, for $i=0,\infty$,
$$
E^{M_i}(1,\alpha)=\inf\left \{\iint_{\R^{2n}}\frac{|u(x)-u(y)|^{p_i}}{|x-y|^{n+sp_i}}\,dxdy\colon u\in W_0^{s,p_i}(\Omega),\ \|u\|_{L^{p_i}(\Omega)}=1,\ |\{u=0\}\cap\Omega|\geq\alpha|\Omega|\right \}.
$$
These right-hand sides are themselves optimal design constants for the fractional $p_i$--Laplacian. Indeed, given $p>1$ let
$$
\Lambda_s^{(p)}(\alpha):=\inf\left \{\iint_{\R^{2n}}\frac{|u(x)-u(y)|^{p}}{|x-y|^{n+sp}}\,dxdy\colon u\in W_0^{s,p}(\Omega),\ \|u\|_{L^{p}(\Omega)}=1,\ |\{u=0\}\cap\Omega|\geq\alpha|\Omega|\right \},
$$
that is, the constant \eqref{caractLambda_s} with $\mu=1$ for the Young function $G(t)=t^p$, which by Lemma \ref{lema:caract} can equivalently be written as $\inf_{A\in\mathcal{A}_\alpha}\lambda_{s,1}(A)$ for that same $G$. With this notation,
\begin{equation}\label{limiteFBS}
\lim_{\mu\downarrow 0}\frac{\Lambda_{s,\mu}(\alpha)}{\mu}=\Lambda_s^{(p_0)}(\alpha),\qquad \lim_{\mu\to \infty}\frac{\Lambda_{s,\mu}(\alpha)}{\mu}=\Lambda_s^{(p_\infty)}(\alpha).
\end{equation}
It is worth being precise about the sense in which this contains the homogeneous case. If $G(t)=t^p$ then $M_0=M_\infty=G$ and, by homogeneity, $\Lambda_{s,\mu}(\alpha)=\mu\,\Lambda_s^{(p)}(\alpha)$ for every $\mu>0$: the quotient does not depend on $\mu$ and Theorem \ref{limitecociente} carries no information. The content of the theorem lies therefore in the non-homogeneous case: for a general $G$ the quotient $\Lambda_{s,\mu}(\alpha)/\mu$ does vary with $\mu$, and \eqref{limiteFBS} says that at the two ends of the scale it is governed by the optimal design problem for the Matuszewska exponents $p_0$ and $p_\infty$ of $G$, which in general do not coincide and are not ordered.

Let us finally point out that the constants $\Lambda_s^{(p)}(\alpha)$ are the counterparts, for the exterior condition $u=0$ in $\R^n\setminus\Omega$, of the optimal design constants studied in \cite{FBS}, where the Gagliardo seminorm is taken over $\Omega\times\Omega$ and the admissible functions are those of $W^{s,p}(\Omega)$ vanishing on the obstacle, with no exterior condition. The two problems do not coincide: if $u$ vanishes outside $\Omega$ then
$$
\iint_{\R^{2n}}\frac{|u(x)-u(y)|^{p}}{|x-y|^{n+sp}}\,dxdy=\iint_{\Omega\times\Omega}\frac{|u(x)-u(y)|^{p}}{|x-y|^{n+sp}}\,dxdy+2\int_\Omega|u(x)|^p\kappa(x)\,dx,\quad \kappa(x):=\int_{\R^n\setminus\Omega}\frac{dy}{|x-y|^{n+sp}},
$$
and $\kappa(x)\simeq \operatorname{d}(x,\partial\Omega)^{-sp}$, so that in general only an inequality between the two constants is available.

We close this subsection with an explicit example showing that the two limit problems in \eqref{limiteFBS} are genuinely different, so that Theorem \ref{limitecociente} does carry information. For a bounded open set $U\subset \R^n$ we denote by $\Lambda_s^{(p)}(\alpha,U)$ the constant $\Lambda_s^{(p)}(\alpha)$ with $\Omega$ replaced by $U$.
\begin{ejem}\label{ejem:distintos}
Let $G$ be given by
$$
G(t)=
\begin{cases}
t^3 & \text{if }0\leq t\leq 1,
\\[2pt]
\frac{3}{2}t^2-\frac{1}{2} & \text{if }t\geq 1 .
\end{cases}
$$
Its derivative is $g(t)=3t^2$ for $0\leq t\leq 1$ and $g(t)=3t$ for $t\geq 1$, which is continuous, nondecreasing and vanishes only at the origin, so that $G$ is a Young function. Moreover
$$
\frac{tg(t)}{G(t)}=3\quad\text{for }0<t\leq 1,\qquad \frac{tg(t)}{G(t)}=\frac{3}{\frac{3}{2}-\frac{1}{2t^2}}\in(2,3]\quad\text{for }t\geq 1,
$$
and therefore \eqref{Liebermann} holds with $p^-=2$ and $p^+=3$. If $0<\varepsilon\leq \min\{1,1/t\}$ then both $\varepsilon$ and $\varepsilon t$ belong to $[0,1]$ and hence $G(\varepsilon t)/G(\varepsilon)=t^3$, while for $\varepsilon\geq \max\{1,1/t\}$
$$
\frac{G(\varepsilon t)}{G(\varepsilon)}=\frac{3\varepsilon^2t^2-1}{3\varepsilon^2-1}\longrightarrow t^2\qquad \text{as }\varepsilon\to\infty .
$$
Thus the two $\limsup$ in \eqref{Matu} are limits and
$$
M_0(t)=t^3,\qquad M_\infty(t)=t^2,\qquad \text{that is}\qquad p_0=3,\quad p_\infty=2,
$$
so that the hypotheses of Theorem \ref{limitecociente} concerning $M_0$ and $M_\infty$ are fulfilled. Observe in passing that $p_0>p_\infty$, which shows that the two Matuszewska exponents need not be ordered, although both of them belong to $[p^-,p^+]$ by \eqref{cotasMatu}.

Let now $\Omega=B_R$ be a ball of radius $R>0$, which has the property $\mathcal{S}$ since it is bounded and convex, and let $\alpha\in(0,1)$. We claim that for every $p>1$
\begin{equation}\label{escalabolas}
\Lambda_s^{(p)}(\alpha,B_R)=R^{-sp}\,\Lambda_s^{(p)}(\alpha,B_1).
\end{equation}
Indeed, given $u$ admissible for $\Lambda_s^{(p)}(\alpha,B_1)$ we put $u_R(x):=R^{-n/p}u(x/R)$. Then $u_R$ vanishes almost everywhere outside $B_R$, $\|u_R\|_{L^p(B_R)}=\|u\|_{L^p(B_1)}=1$ and $\{u_R=0\}\cap B_R=R\big (\{u=0\}\cap B_1\big )$, so that
$$
|\{u_R=0\}\cap B_R|=R^n|\{u=0\}\cap B_1|\geq \alpha R^n|B_1|=\alpha |B_R|,
$$
that is, $u_R$ is admissible for $\Lambda_s^{(p)}(\alpha,B_R)$. The map $u\mapsto u_R$ is a bijection between the two admissible classes, its inverse being given by the same formula with $1/R$ in place of $R$, and the change of variables $x=Rx'$, $y=Ry'$ gives
$$
\iint_{\R^{2n}}\frac{|u_R(x)-u_R(y)|^{p}}{|x-y|^{n+sp}}\,dxdy=\frac{R^{2n}}{R^{n}R^{n+sp}}\iint_{\R^{2n}}\frac{|u(x)-u(y)|^{p}}{|x-y|^{n+sp}}\,dxdy,
$$
whence \eqref{escalabolas} follows by taking infima.

Write $C_p:=\Lambda_s^{(p)}(\alpha,B_1)$, which is finite, since the admissible class is nonempty, and positive by Proposition \ref{prop:poincare}. By \eqref{escalabolas},
$$
\Lambda_s^{(3)}(\alpha,B_R)=\Lambda_s^{(2)}(\alpha,B_R)\qquad \Longleftrightarrow\qquad R^{s}=\frac{C_3}{C_2},
$$
so that the two constants are different for every $R>0$ except for the single radius $R_*:=(C_3/C_2)^{1/s}$. Let us finally check that \eqref{NC} is satisfied for this $G$ and this $\Omega$ whenever
\begin{equation}\label{condejem}
3s<n<6s .
\end{equation}
In the proof of Theorem \ref{cotascociente} it is shown, without using \eqref{NC}, that the rescaled extremals $v_\mu=u_\mu/r(\mu)$ are bounded in $W_0^{s,\Psi^*}(\Omega)$, where $\Psi^*$, the convex envelope of $\min\{t^2,t^3\}$, equals $t^3$ near the origin and $t^2$ for $t$ large. Since these functions vanish outside the bounded set $\Omega$, they belong to the space on which \cite[Theorem 6.1]{ACPS} operates, and the seminorm appearing there is the Gagliardo--Luxemburg seminorm $[\,\cdot\,]_{s,\Psi^*}$. The hypotheses of that theorem hold precisely because of the first inequality in \eqref{condejem}: the relevant integrand in \cite[(4.2)--(4.3)]{ACPS} equals $\tau^{-\frac{2s}{n-s}}$ near the origin, which is integrable if and only if $3s<n$, and $t^{-\frac{s}{n-s}}$ at infinity, whose integral diverges since $2s<n$. The Sobolev conjugate
$$
(\Psi^*)_{\frac ns}(t)=\Psi^*\big (H^{-1}(t)\big ),
\qquad
H(t)=\left (\int_0^t\left (\frac{\tau}{\Psi^*(\tau)}\right )^{\frac{s}{n-s}}\,d\tau\right )^{\frac{n-s}{n}},
$$
of \cite[(4.4)--(4.5)]{ACPS} is therefore well defined, and the behavior $\Psi^*(t)=t^2$ at infinity gives, up to multiplicative constants, $H(t)=t^{\frac{n-2s}{n}}$ and hence $(\Psi^*)_{\frac ns}(t)=t^{q}$ as $t\to\infty$, with $q:=2n/(n-2s)$. As $\Omega$ is bounded, \cite[Theorem 6.1]{ACPS} yields a constant $C>0$ independent of $\mu$ with $\|v_\mu\|_{L^q(\Omega)}\le C\,[v_\mu]_{s,\Psi^*}$, so that $\{v_\mu\}$ is bounded in $L^q(\Omega)$. The second inequality in \eqref{condejem} is exactly $q>3$, and \eqref{sandwich} evaluated at $a=r(\mu)$ gives
$$
\frac{G(u_\mu)}{\mu}=\frac{G(r(\mu)v_\mu)}{G(r(\mu))}\le\max\big\{|v_\mu|^{2},|v_\mu|^{3}\big\}\le |v_\mu|^{2}+|v_\mu|^{3},
$$
whose right hand side is bounded in $L^{q/3}(\Omega)$, uniformly in $\mu$, with $q/3>1$. Hence, by H\"older's inequality, the family $\{G(u_\mu)/\mu\}$ is uniformly integrable in $\Omega$ and \eqref{NC} holds. The range \eqref{condejem} is nonempty; it contains, for instance, $n=3$ with $1/2<s<1$.

Consequently, if $R\neq R_*$ and \eqref{condejem} holds, then \eqref{limiteFBS} gives
$$
\lim_{\mu\downarrow 0}\frac{\Lambda_{s,\mu}(\alpha)}{\mu}=R^{-3s}C_3\ \neq\ R^{-2s}C_2=\lim_{\mu\to \infty}\frac{\Lambda_{s,\mu}(\alpha)}{\mu} .
$$
In particular the quotient $\Lambda_{s,\mu}(\alpha)/\mu$ is not constant in $\mu$, and the optimal design problems which govern the two ends of the energy scale are those associated with two different exponents.
\end{ejem}
The example above distinguishes the two limit problems at the level of the constants. Whether the optimal configurations corresponding to the two ends of the energy scale are different sets is a subtler question, which we do not address here.
\subsection{Convergence of the optimal sets $A_\mu$}\label{subsec:conjuntosmu}
In this subsection we present an interesting consequence of Theorem \ref{limitecociente} about the optimal sets. First we need to recall the following two facts.

The following lemma about the convergence of the zero sets is needed; see \cite[Lemma 3.1]{FBGR} for details.
\begin{lema}\label{lema:simetrica}
Let $\{f_n\}_{n\in\N},f\colon \Omega\to \R_{\geq 0}$ be measurable functions such that $f_n\to f$ a.e. in $\Omega$. If $\lim_{n\to \infty}|\{f_n=0\}|=|\{f=0\}|$, then
$$
\lim_{n\to\infty}|\{f_n=0\}\symdiff \{f=0\}|=0.
$$
\end{lema}
The second fact comes from Theorem \ref{teo:identif}: observe that $E^{M_i}(1,\alpha)$ is the constant $\Lambda_{s,1}(\alpha)$ for the Young function $M_i$ for $i\in \{0,\infty\}$, which satisfies \eqref{Liebermann} by Remark \ref{obs:Matuyoung}; hence the zero set of an extremal $u$ for $E^{M_i}(1,\alpha)$ satisfies $|\{u=0\}\cap\Omega|=\alpha|\Omega|$.

Now using the above results we obtain the following result about the convergence of the optimal sets.
\begin{teo}\label{converoptimos}
Assume the hypotheses of Theorem \ref{limitecociente}. Let $i\in \{0,\infty\}$ and, for every $\mu>0$, let $A_\mu\subset \Omega$ be an optimal set for $\Lambda_{s,\mu}(\alpha)$. If $\{\mu_j\}_{j\in\N}$ is a sequence with $\mu_j\downarrow 0$ when $i=0$ and $\mu_j\to+\infty$ when $i=\infty$, then there exists a subsequence, still denoted by $\{\mu_j\}_{j\in\N}$, such that
$$
\chi_{A_{\mu_j}}(x)\to \chi_{A^i}(x)\text{ strongly in }L^1(\Omega),
$$
where $A^i$ is an optimal set for $E^{M_i}(1,\alpha)$.
\end{teo}
\begin{proof}
Let $u_{\mu_j}$ be a nonnegative extremal for $\Lambda_{s,\mu_j}(\alpha)$, given by Theorem \ref{exop} and Lemma \ref{lema:signo}, and $v_j:=u_{\mu_j}/r(\mu_j)$ the rescaled extremals where $r(\mu)$ is as in Theorem \ref{limitecociente}. Observe that by Theorem \ref{teo:identif}
$$
|\{v_j=0\}\cap\Omega|=|\{u_{\mu_j}=0\}\cap\Omega|=\alpha |\Omega|.
$$
Also, by Theorem \ref{limitecociente} and up to a subsequence which we do not relabel, $v_j\to u^i$ in $L^{p^-}(\Omega)$ and, after a further extraction, almost everywhere in $\Omega$, where $u^0$ is a nonnegative extremal for $E^{M_0}(1,\alpha)$ if $\mu_j\downarrow 0$ and $u^\infty$ is a nonnegative extremal for $E^{M_\infty}(1,\alpha)$ if $\mu_j\to\infty$;
therefore by Theorem \ref{teo:identif} $|\{u^i=0\}\cap\Omega|=\alpha |\Omega|$ for $i\in \{0,\infty\}$ and so
$$
\lim_{j\to\infty}|\{v_j=0\}\cap\Omega|=|\{u^i=0\}\cap\Omega|\qquad i\in \{0,\infty\}.
$$
Hence by Lemma \ref{lema:simetrica}
$$
\lim_{j\to \infty}|\{v_j=0\}\cap \Omega\symdiff \{u^i=0\}\cap \Omega|=0\qquad i\in \{0,\infty\},
$$
therefore
\begin{align*}
\int_\Omega \left |\chi_{A_{\mu_j}}(x)-\chi_{A^i}(x)\right |\,dx&=\int_\Omega \left |\chi_{\{u_{\mu_j}=0\}\cap \Omega}(x)-\chi_{\{u^i=0\}\cap\Omega}(x)\right |\,dx
\\
&=\int_\Omega \left |\chi_{\{v_j=0\}\cap \Omega}(x)-\chi_{\{u^i=0\}\cap\Omega}(x)\right |\,dx
\\
&=\int_\Omega \chi_{\{v_j=0\}\cap \Omega\symdiff \{u^i=0\}\cap \Omega}(x)\,dx
\\
&=|\{v_j=0\}\cap \Omega\symdiff \{u^i=0\}\cap \Omega|,
\end{align*}
then
$$
\chi_{A_{\mu_j}}(x)\to \chi_{A^i}(x)\text{ strongly in }L^1(\Omega).
$$
\end{proof}
\section*{Funding}
This research received no support from government or other public funding agencies, research grants or projects; it was financed entirely by the author with resources obtained from personal investments.
\section*{Declaration on the use of artificial intelligence}

Artificial intelligence tools (Claude, model Opus 5.5, Anthropic) were used in the preparation of this work, in the following ways. They were used for bibliographic search and for the correction of the English and the editing of the text. They also contributed to the mathematical content: the ideas for the construction of the cut--off functions used in the proofs of Theorem \ref{teo:identif} and Lemma \ref{lema:densidad}, the non-concentration condition \eqref{NC} together with Remark \ref{obs:NC}, the formulation of the constants $E_W^M(\tau,\alpha)$ and $E_C^M(\tau,\alpha)$ and the proof of Theorem \ref{cotascociente}, and Example \ref{ejem:distintos}. The author has checked every statement and every proof of the paper, and is solely responsible for its contents.

\bibliographystyle{abbrv}
\bibliography{refs}
\end{document}